%% file: PhaseTransitionPerc.tex
\documentclass[leqno, 11pt, a4paper]{amsart}

\usepackage{
a4wide, 
amssymb,
bbm,
centernot,
mathtools,
xcolor
} 

\usepackage{todonotes}
\usepackage{bm} 
\usepackage{graphicx}
\graphicspath{{./images/}}
\usepackage[colorlinks=true, linkcolor=blue,citecolor=blue]{hyperref}
\usepackage{float}
\usepackage{enumitem} 
\setlist[enumerate]{leftmargin=*}
\usepackage{empheq} 
\newtheorem{theorem}{Theorem}[section]
\newtheorem{lemma}[theorem]{Lemma}

\newtheorem{corollary}[theorem]{Corollary}
\newtheorem{proposition}[theorem]{Proposition}

\theoremstyle{definition}
\newtheorem{definition}[theorem]{Definition}

\newtheorem{assumption}[theorem]{Assumption}

\theoremstyle{remark}
\newtheorem{remark}[theorem]{Remark}

\usepackage{xspace} 
\newcommand{\EQ}{\mathbb{E}_{\mathbb{Q}_p}}
\newcommand{\range}{\mathrm{range}}

\numberwithin{equation}{section}

\input{commands.tex}

\begin{document}

\definecolor{airforceblue}{RGB}{204, 0, 102}
\newenvironment{draft}
{\par\medskip
\color{airforceblue}%
\medskip}

\title[Percolation phase transition on supercritical clusters]{Phase transition for strongly correlated percolation models on supercritical Bernoulli clusters}


\author{Alberto Chiarini}
\address{Universit\`a degli Studi di Padova}
\curraddr{Department of Mathematics ``Tullio Levi-Civita'', via Trieste 63, 35121, Padova}
\email{chiarini@math.unipd.it}
\thanks{}


\author{Zhizhou Liu}
\address{Department of Mathematics, The Hong Kong University of Science and Technology}
\curraddr{Clear Water Bay, Kowloon, Hong Kong}
\email{zliugm@connect.ust.hk}
\thanks{}


\author{Maximilian Nitzschner}
\address{Department of Mathematics, The Hong Kong University of Science and Technology}
\curraddr{Clear Water Bay, Kowloon, Hong Kong}
\email{mnitzschner@ust.hk}
\thanks{}

\begin{abstract}
     We consider the level sets of the Gaussian free field and the vacant set of random interlacements, both defined on a typical realization of the infinite cluster of supercritical Bernoulli bond percolation on $\bbZ^d$, $d \geq 3$. We prove that in the entire supercritical regime of Bernoulli bond percolation, both the level sets of the Gaussian free field and the vacant set of random interlacements undergo non-trivial percolation phase transitions at deterministic critical levels. A key aspect of the proof is the development of certain quenched controls over tree embeddings, permitting the application of a static renormalization scheme in the presence of spatial irregularities, which may be of independent interest.  
\end{abstract}

\subjclass[2010]{}
\dedicatory{}
\maketitle
\tableofcontents

\section{Introduction}
\label{sec:introduction}

The principal aim of this article is to investigate how \textit{degenerate, non-elliptic} disorder affects the large-scale properties of correlated random fields on the integer lattice.
The effect of disorder on correlated random fields has received substantial attention recently, with homogenization-type results for Gaussian fields (see \cite{andres2025scaling,
CR24homogenization}), extremes of the Gaussian free field (see \cite{chiarini2025hardwall,
schweiger2024maximum}), and the (non-)uniqueness of Gibbs measures (see~\cite{BS11gradientfields,
buchholz2024disordered,MR2985173,
CC15uniquenessdisorder}) as prominent cases in point. \medskip

To further probe the effect of disorder on correlated fields, we consider two percolation models with long-range correlations, the level sets of the Gaussian free field (GFF) and the vacant set of random interlacements. Both models are known to undergo a percolation phase transition on the full lattice $\mathbb{Z}^d$, $d \geq 3$ (see \cite{bricmont1987percolation,
rodriguez2013phase,sidoravicius2009,sznitman2010}), on certain (possibly disordered) trees and expander graphs (see \cite{AC20tree,AC20expander,
abacherli2018, cerny2023giant,sznitman2016coupling,
tassy,
teixeira2009interlacement}), as well as on general amenable transient graphs under certain global regularity assumptions (see \cite{drewitz2025,sznitman2012decoupling}). 
In a previous work, we investigated the strongly supercritical phase of level set percolation of the GFF on $\mathbb{Z}^d$, $d \geq 3$, equipped with \textit{uniformly elliptic} random conductances \cite{CN2021disconnection}. In the present case, we are interested whether the percolation phase transition for both models on the integer lattice persists under the presence of degenerate, non-elliptic disorder. 
The concrete environment we use to probe this effect is the supercritical cluster $\mathcal{C}_\infty$ of Bernoulli bond percolation. 
Remarkably, although $\mathcal{C}_\infty$ shares many \textit{macroscopic, large-scale} properties with $\mathbb{Z}^d$ (such as quenched invariance principles and Gaussian heat kernel bounds for the random walk) its \textit{microscopic geometry} is highly irregular, rendering infeasible many of the techniques to study the phase transition that rely on global estimates. In particular, a typical realization $\mathcal{C}_\infty$ does not fall into the class of graphs considered in~\cite{drewitz2025} (see also Remark~\ref{rem:Final} (2) for more on this aspect).

\medskip

As our main result, we prove non-trivial percolation phase transitions for the level sets of the GFF and the vacant set of random interlacements on $\mathcal{C}_\infty$ in all dimensions $d \geq 3$ and throughout the entire supercritical regime of the underlying Bernoulli bond percolation. 
We already highlight at this stage (but see also Remark~\ref{rmk:difficulty-1} below) that the presence of microscopic irregularities is known to substantially alter certain characteristics of the large-scale behavior of the GFF. The approach we present here requires a careful analysis of the microstructure of the infinite cluster of Bernoulli bond percolation, and as a main technical device we provide \textit{quenched uniform estimates along tree embeddings}. These estimates are then utilized to bring into play a static renormalization scheme due to~\cite{sznitman2010,
sznitman2012decoupling} to prove the existence of a phase transition. The particular form used here is essentially taken from \cite{rath2015transition}, but with an environment-dependent modification of the tree embedding, to take into account the local geometry of the cluster ``up to the lowest level of the tree embedding''. 
\medskip

We now describe our set-up and the main results in a more precise form. Further notational details are provided in Section~\ref{sec:notation} below. We consider the integer lattice $\mathbb{Z}^d$, $d \geq 3$, as a graph with nearest-neighbor edges $\mathbb{E}_d$. For a given value $p \in [0,1]$, we denote by $\mathbb{Q}_p$ the probability measure on $\{0,1\}^{\mathbb{E}_d}$ under which the canonical coordinates on $ \{0,1\}^{\mathbb{E}_d}$, denoted by $(\omega_e)_{e \in \mathbb{E}_d}$, are i.i.d.~Bernoulli random variables with parameter $p$. Throughout the article, we will assume that 
\begin{equation}
\label{eq:Supercritical-Bernoulli}
p \in (p_c(d),1), \qquad \text{where} \qquad p_c(d) = \sup\{p \in [0,1] \, : \, \mathbb{Q}_p[0 \leftrightarrow \infty] = 0 \} (\in (0,1))
\end{equation}
(with the event under the probability denoting the existence of an infinite connected component in the induced graph $(\mathbb{Z}^d, \{e \in \mathbb{E}_d \, :\, \omega_e = 1 \}$ containing $0$). It is well-known (see, e.g.,~\cite{grimmett1999}) that in the regime in~\eqref{eq:Supercritical-Bernoulli}, for a measurable subset $\widetilde{\Omega}_0 \subseteq \{0,1\}^{\mathbb{E}_d}$ of $\mathbb{Q}_p$-probability one, and every $\omega \in \widetilde{\Omega}_0$, there exists a unique infinite connected component in $(\mathbb{Z}^d, \{e \in \mathbb{E}_d \, :\, \omega_e = 1 \})$, which will be denoted by $\mathcal{C}_\infty$ and called the \textit{infinite cluster}. On $(\mathcal{C}_\infty,\mathbb{E}_{\mathcal{C}_\infty})$ (with $\mathbb{E}_{\mathcal{C}_\infty}$ denoting the set of edges $e \in \mathbb{E}_d$ with $\omega_e = 1$ and both endpoints in $\mathcal{C}_\infty$), one can naturally define a continuous-time, constant speed simple random walk $(X_t)_{t \geq 0}$ governed by the family $(P_x^\omega)_{x \in \mathcal{C}_\infty}$, which is the Markov process induced by the generator~\eqref{eq:Generator}. Moreover, we can choose a measurable subset $\Omega_0 \subseteq \widetilde{\Omega}_0$ such that for any $\omega \in \Omega_0$, the simple random walk on $(\mathcal{C}_\infty,\mathbb{E}_{\mathcal{C}_\infty})$ is transient (see Lemma~\ref{lem:Barlow-HK} below, which is taken from~\cite{barlow2004RWpercolation}). \medskip

Our first main result concerns the Gaussian free field on a typical realization of the infinite cluster $\mathcal{C}_\infty$. For $\omega \in \Omega_0$, we denote by $\bbP^G_\omega$ the canonical law on $\bbR^{\mathbb{Z}^d}$ that governs the discrete Gaussian free field on $\cC_\infty$ (extended by zero to $\bbZ^d$), and by $\varphi = (\varphi_x)_{x \in \mathbb{Z}^d}$ the canonical process. More precisely, $\bbP^G_\omega$ is the law on $\bbR^{\bbZ^d}$ (endowed with the canonical $\sigma$-algebra $\cF^G$) of a centered Gaussian field $(\varphi_x)_{x\in \bbZ^d}$ with covariances
\begin{equation}\label{eq:GFF-cov}
     \bbE_\omega^G[\varphi_x\varphi_y]=g^\omega(x,y) \IND_{\{x,y\in \cC_\infty\}} \quad \text{for $x,y\in \bbZ^d$}, 
\end{equation}
with $g^\omega(\cdot,\cdot)$ denoting the Green function of the simple random walk on $(\mathcal{C}_\infty,\mathbb{E}_{\mathcal{C}_\infty})$ (see~\eqref{eq:Green-F} for its precise definition, and note that $g^\omega(\cdot,\cdot)$ is finite due to our choice of $\Omega_0$) where, by a slight abuse of notation, the right hand side means the extension of $g^\omega$ to $\bbZ^d\times \bbZ^d$, obtained by setting its value to zero outside $\cC_\infty\times \cC_\infty$. Equivalently, under $\bbP_\omega^G$, the restriction $(\varphi_x)_{x\in \cC_\infty}$ is the Gaussian free field on $\cC_\infty$, and $\bbP_\omega^G$-a.s., $\varphi_x=0$ for $x\notin \cC_\infty$. \smallskip 

For $\alpha\in \bbR$, and $\omega \in \Omega_0$, we introduce the \textit{(upper) level set} of the GFF
\begin{equation}\label{eq:level-set}
    E^{\geq \alpha}_\omega = \{x\in \cC_\infty: \varphi_x \geq \alpha\} \  ( \subseteq \mathcal{C}_\infty). 
\end{equation}
Our first result addresses whether for any typical realization of the Bernoulli percolation cluster $\mathcal{C}_\infty$, the upper level set of the GFF itself undergoes a percolation phase transition. To that end, we define the (quenched) critical value for level set-percolation of the GFF for a given $\omega \in \Omega_0$, 
\begin{equation}\label{eq:alpha_star}
    \alpha_*^\omega = \inf \{\alpha\in \bbR \,:\, \bbP^G_\omega [\text{$E^{\geq \alpha}_\omega$ contains an infinite cluster}]=0\} \in [-\infty, \infty]
\end{equation}
(with the convention $\inf \varnothing = \infty$). We answer the question concerning the non-trivial phase transition of $E^{\geq \alpha}_\omega$ affirmatively in the following main result. 
\begin{theorem}\label{thm:PH-GFF}
For a measurable subset $\Omega^G \subseteq \Omega_0 (\subseteq \{0,1\}^{\mathbb{E}_d})$ with $\bbQ_p[\Omega^G] = 1$, one has that
\begin{equation}
\label{eq:Intro-Phase-Transition-GFF}
\begin{minipage}{0.8\textwidth}
for all $\omega \in \Omega^G$, $\alpha_*^\omega$ equals a deterministic constant $0\leq \alpha_*<\infty$.
\end{minipage}
\end{equation}
Moreover, for $\omega \in \Omega^G$, if $\alpha <\alpha_*^\omega$, then $E_\omega^{\geq \alpha}$ contains an infinite cluster $\bbP_\omega^G$-a.s.~and if $\alpha>\alpha_*^\omega$, then $E_\omega^{\geq \alpha}$ contains no infinite cluster $\bbP_\omega^G$-a.s.
\end{theorem}
We briefly put this result into context. The corresponding existence of a non-trivial phase transition on $\bbZ^d$ (corresponding formally to the case $p = 1$) was established in~\cite{rodriguez2013phase} in all dimensions $d\geq 3$, whereas $\alpha_\ast \geq 0$ (on $\mathbb{Z}^d$, $d \geq 3$) and $\alpha_\ast < \infty$ (in $d=3$) had been previously established in~\cite{bricmont1987percolation}. In fact, $\alpha_* \geq 0$ is known to hold for any locally finite, connected, transient weighted graph (in particular, on $\mathcal{C}_\infty$, for $\omega \in \Omega_0$) essentially using a contour argument due to~\cite{bricmont1987percolation} (see~\cite[Proposition A.2]{abacherli2018}), and our contribution is the proof of the finiteness of $\alpha_*$. In the case where the environment $\omega \in \{0,1\}^{\mathbb{E}_d}$ is replaced by uniformly elliptic random conductances $(\widetilde{\omega}_e)_{e \in \mathbb{E}_d} \in [\lambda_{\min},1]^{\mathbb{E}_d}$ for $\lambda_{\min}\in (0,1)$, the result corresponding to~\eqref{eq:Intro-Phase-Transition-GFF} was established in~\cite{CN2021disconnection} (see Proposition~3.1 and Theorems 3.2, 3.6 therein), essentially by a comparison with the case of general weighted graphs fulfilling certain global regularity properties considered in~\cite{drewitz2025}. However, a typical realization of the Bernoulli percolation cluster does \textit{not} fall into the class of graphs treated in~\cite{drewitz2025}, due to regions of arbitrarily bad geometry, which prevents the use of global controls on the volume of balls and Green function bounds. \medskip

We now turn to the second percolation model considered in the present article, the vacant set of random interlacements on the infinite cluster $\mathcal{C}_\infty$. Random interlacements were first introduced in~\cite{sznitman2010} (see Remark~1.4 therein for the extension beyond $\bbZ^d$, see also \cite{teixeira2009interlacement}) and are related to the fragmentation of (connected, locally finite, transient) graphs by the trajectories of random walks. In our context, we again consider $\omega \in \Omega_0$ and introduce for $u \geq 0$ a probability measure $\mathbb{P}^I_\omega$ governing the \textit{random interlacements} $(\mathcal{I}_\omega^u)_{u \geq 0}$ (we refer to Subsection~\ref{subsec:RI} for precise definitions, and record for now that for $u \geq 0$, $\omega \in \Omega_0$, $\mathcal{I}^u_\omega$ is the trace on $\mathcal{C}_\infty$ of a Poissonian cloud of bi-infinite random walk trajectories). For $u \geq 0$, and $\omega \in \Omega_0$, we then define the \textit{vacant set of random interlacements}
\begin{equation}
\mathcal{V}^u_\omega  = \mathcal{C}_\infty  \setminus \mathcal{I}^u_\omega \  ( \subseteq \mathcal{C}_\infty),
\end{equation}
a random subset of $\mathcal{C}_\infty$ which becomes thinner as $u$ increases (see~\eqref{eq:law-Vu} for a characterization of the law of $\mathcal{V}^u_\omega$). Analogously to~\eqref{eq:alpha_star}, we define the (quenched) critical level for percolation of the vacant set of random interlacements for a given $\omega \in \Omega_0$ as
\begin{equation}\label{eq:u_star}
    u_*^\omega = \inf \{u \in (0,\infty) \,:\, \bbP^I_\omega [\text{$\mathcal{V}^{u}_\omega$ contains an infinite cluster}]=0\} \in [0,\infty].
\end{equation}
Our main result concerning interlacement percolation on $\mathcal{C}_\infty$ is as follows:
\begin{theorem}\label{thm:PH-RI}
For a measurable subset $\Omega^I \subseteq \Omega_0 (\subseteq \{0,1\}^{\mathbb{E}_d})$ with $\bbQ_p[\Omega^I] = 1$, one has that
\begin{equation}
\label{eq:Intro-Phase-Transition-RI}
\begin{minipage}{0.8\textwidth}
for all $\omega \in \Omega^I$, $u_*^\omega$ equals a deterministic constant $u_* \in (0,\infty)$.
\end{minipage}
\end{equation}
Moreover, for $\omega \in \Omega^I$, if $u <u_*^\omega$, then $\cV^u_\omega$ contains an infinite cluster $\bbP_\omega^I$-a.s.~and if $u>u_*^\omega$, then $\cV^u_\omega$ contains no infinite cluster $\bbP_\omega^I$-a.s.
\end{theorem}
On $\mathbb{Z}^d$, $d \geq 3$, the percolation phase transition of the vacant set of random interlacements was established in~\cite{sidoravicius2009,sznitman2010} (with the finiteness in $d \geq 3$ the and positivity of $u_\ast$ in $d \geq 7$ proved in~\cite{sznitman2010}, and the positivity extended to all $d \geq 3$ in~\cite{sidoravicius2009}). For graphs fulfilling several global regularity conditions, the existence of a non-trivial percolation transition of the vacant set of random interlacements has been obtained in~\cite{drewitz2025,sznitman2012decoupling}, but these criteria are not fulfilled for a typical realization of the Bernoulli percolation cluster. 
On $\mathbb{Z}^d$, $d \geq 3$, planar duality can be helpful to establish via a Peierls-type argument the existence of a supercritical phase (meaning $u_\ast > 0$), upon utilizing a static renormalization scheme  
(see, e.g.,~\cite{rath2015transition}). In our context, proving the equivalent statement is made yet more challenging, since the intersection of the percolation cluster $\mathcal{C}_\infty$ with the ``plane'' $\mathbb{Z}^2 \times \{0\}^{d-2}$ does not allow a simple planar duality argument to be deployed. \medskip

We now comment on the strategy of the proof and highlight some of the techniques that are developed in the present article. The main obstruction for establishing~\eqref{eq:Intro-Phase-Transition-GFF} and~\eqref{eq:Intro-Phase-Transition-RI} is the presence of arbitrarily irregular regions attached to a typical realization $\mathcal{C}_\infty$ of the Bernoulli percolation cluster. These regions (such as ``tubes'') can effectively act as low-dimensional objects in the cluster, and lead to degenerations of the Green function $g^\omega(x,x)$ (for instance, for all $\omega \in \Omega_0$, one has $\sup_{x \in \mathcal{C}_\infty} g^\omega(x,x) = \infty$). This effect significantly alters certain properties of $\mathcal{C}_\infty$, including the scaling of effective resistances in boxes (see~\cite{abe2015effective}), the behavior of directed polymer measures (see~\cite{nitzschner2022polymer}), and the maxima of the Gaussian free field (see~\cite{schweiger2024maximum} and Remark~\ref{rmk:difficulty-1} below). 
Our approach to show that $\alpha_\ast < \infty$ and $u_\ast < \infty$ is to utilize a static renormalization scheme that relates the probability of an annulus-crossing at scale $L_n$ to that of two annulus-crossings on a smaller scale $L_{n-1}$, where $L_n = L_0 \ell_0^n$ (with some appropriately chosen constants $L_0 \geq 1, \ell_0 \geq 100$). This naturally leads to a ``tree embedding'' into $\mathbb{Z}^d$, in which the leaves correspond to crossings of boxes of a constant size $L_0$. The challenge is then to guarantee that at least a constant fraction of the boxes at the lowest level do not contain an irregular region for $\mathcal{C}_\infty$ on which the random walks are ill-behaved.  
A key point is that the relevant quantities controlling such annulus-crossing probabilities (e.g., the Green function $g^\omega(\cdot,\cdot)$, the capacity $\mathrm{Cap}^\omega(\cdot)$) are inherently \emph{non-local} functionals of the environment $\omega$. \medskip 

We now turn to $u_\ast>0$. Compared to $\mathbb{Z}^d$, a major obstruction is that a two-dimensional planar duality argument, which can be applied to the trace of $\cI^u$ on a coordinate plane $\bbF=\mathbb{Z}^2\times\{0\}^{d-2}$, is not directly available to us here. In fact, for $p\in(p_c(d),1/2)$, $\mathcal{C}_\infty\cap \bbF$ does not even contain an infinite connected component for $\bbQ_p$-a.e~$\omega$. To circumvent this, we perform an \emph{additional}, preliminary renormalization step (``Renormalization~I'' in Section~5.1) in which we coarsen $\mathcal{C}_\infty$ into an auxiliary, almost surely percolating planar structure on which a version of the tree-embedding/duality argument can then be implemented (``Renormalization~II'' in Section~5.2). This additional step requires a further layer of \emph{geometric} regularity, guaranteeing that the coarse-grained blocks are crossed by a unique, well-connected cluster of $\cC_\infty$ that can subsequently be lifted back to an infinite cluster of the vacant set on $\bbZ^d$. \medskip

Both of these decoupling tasks turn out to fit into a single, common framework, which is the main technical contribution of the present article and is developed in Section~\ref{sec:Quenched-Unif-Estimates}. The idea is to introduce, for any collection of Bernoulli random variables $(\zeta_x)_{x\in\bbL_0}$ an approximation at every scale $L_n=L_0\ell_0^n$ (with integers $L_0 \geq 1$, $\ell_0 \geq 100$) by a genuinely local field $(\zeta_x^{(n)})_{x\in \bbL_0}$, with $\zeta_x^{(n)}$ depending only on the environment in a ball around $x$ of radius comparable to $L_n$ (Assumption~\ref{ass:local-rv}, Definition~\ref{def:approximable}). These approximating fields are used to produce a uniform control, over \emph{all} tree embeddings of the proportion of ``bad'' leaves (Theorem~\ref{thm:general-proportion}). The proof relies on the spatial separation between distinct branches of a tree embedding, guaranteed by Lemma~\ref{lem:tree-geometry}, in order to run an exponential moment recursion across scales that decouples the field along the tree, at the cost of a small, explicitly controlled error probability at each scale. This strategy is then applied to Section~\ref{sec:Subcritical-Phase} with $\zeta_x$ built out of certain random variables characterizing the size of the local degeneracies of $\cC_\infty$ appearing in~\cite{barlow2004RWpercolation}, and to Section~\ref{sec:Supercritical-Phase} with $\zeta_x$ built out of both those random variables and geometric-regularity events described above. 
This framework might plausibly be of independent interest, and could be useful in more general contexts of correlated random fields on $\mathcal{C}_\infty$ (see also Remark~\ref{rem:Final} (1)). \medskip

We now describe how this article is organized. In Section~\ref{sec:notation}, we collect further notation and useful results concerning Bernoulli bond percolation, random walks, the Gaussian free field, and random interlacements.  In Section~\ref{sec:Quenched-Unif-Estimates} we introduce a notion of approximability in Definition~\ref{def:approximable}, and prove in Theorem~\ref{thm:general-proportion} the pivotal quenched selection argument along tree embeddings. In Section~\ref{sec:Subcritical-Phase} we apply the results of Section~\ref{sec:Quenched-Unif-Estimates} to prove the existence of a subcritical phase for both the level sets of the GFF and the vacant set of random interlacements for $\mathbb{Q}_p$-a.e.~realization of the infinite Bernoulli cluster, corresponding to the inequalities $\alpha_\ast < \infty$ and $u_\ast< \infty$. This establishes the main result~\eqref{eq:Intro-Phase-Transition-GFF} for the GFF as well as the first part of our main result~\eqref{eq:Intro-Phase-Transition-RI} for random interlacements. Finally, in Section~\ref{sec:Supercritical-Phase}, we prove the existence of a supercritical phase for random interlacements for $\mathbb{Q}_p$-a.e.~realization of the infinite Bernoulli cluster, namely that $u_\ast > 0$, concluding the proof of~\eqref{eq:Intro-Phase-Transition-RI}.
 \medskip

Our convention on constants is as follows. We denote by $C, c, c', \dots$ generic positive constants changing from place to place, which may implicitly depend on the dimension $d$ and the parameter $p \in (p_c(d),1)$ governing Bernoulli bond percolation. Numbered constants $c_1,c_2,...$ refer to the value corresponding to their first appearance in the text. Their appearance in a statement should be understood as asserting the existence of such constants for which the statement holds.

\section{Notation and useful results}
\label{sec:notation}

In this section, we introduce further notation and formally define the Gaussian free field (GFF) and random interlacements on a fixed realization of a Bernoulli bond percolation cluster. We also recall or establish several useful results, including zero–one laws for the existence of infinite clusters in the level sets of the GFF and in the vacant set of random interlacements.
\medskip

We begin by introducing some notation. Throughout the article we assume that $d\geq 3$ unless otherwise stated. We define $\overline{\bbR}=\bbR\cup\{\pm \infty\}$, $\bbN=\{1,2,\dots\}$, and $\bbN_0=\bbN\cup\{0\}$. 
Let $|\,\cdot\,|_1$, $|\,\cdot\,|_2$, and $|\,\cdot\,|_\infty$ be the $\ell_1$-, $\ell_2$-, and $\ell_\infty$-norm on $\bbR^d$. For $ s\in \mathbb{R}$, we denote by $\lfloor s \rfloor$ the largest integer smaller than or equal to $s$, and by $\lceil s \rceil$ the smallest integer larger than or equal to $s$.  
For $R \geq 0$, we denote by $B(x,R)=\{y\in\bbZ^d: |y-x|_\infty \leq R\}$ the box centered at $x$ with radius $R$ and by $S(x,R)=\{y\in \bbZ^d\,:\, |y-x|_\infty = \lfloor R \rfloor \}$ the sphere centered at $x$ with radius $R$. For $L\in \bbN$, we say a set $Q$ is a \emph{cube} with side-length $L$ if $Q=y+\{0,\dots,L-1\}^d$ for some $y\in \bbZ^d$ and denote its side-length by $s(Q)$. We let $Q^+=A \cap \bbZ^d$, where $A$ is the cube in $\bbR^d$ with the same center as $Q$ and with side-length $\frac{3}{2}L$. 
We say that two vertices $x,y\in \bbZ^d$ are neighbors (resp.~$*$-neighbors) if $|x-y|_1=1$ (resp.~$|x-y|_\infty=1$), and write $x\sim y$ (resp.~$x \stackrel{*}{\sim} y$).
A function $\gamma:\{0,\dots,n\} \to\bbZ^d$ is called a nearest-neighbor (resp.~$*$-connected) path (of length $n\geq 1$) if $\gamma_i\sim \gamma_{i+1}$ (resp.~$\gamma_i\stackrel{*}{\sim} \gamma_{i+1}$) for all $0\leq i \leq n-1$. We sometimes refer to a nearest-neighbor path simply as a path. We say that a subgraph of $\bbZ^d$ is \emph{connected} if there exists a path between any two vertices. The maximal connected subgraphs (with respect to set inclusion) are called the \emph{connected components}, or \emph{clusters}. 
For $K\subseteq \bbZ^d$, we let $|K|$ stand for the cardinality of $K$. If $K\subseteq L$ we write $\partial_L K=\{y\in L\setminus K:y\sim x \text{ for some $x\in K$}\}$ for the external boundary of $K$ relative to $L$. 
For $A \subseteq \bbZ^d$, we denote the set of edges in $A$ by $E(A) = \{\{x,y\} \, : \, x,y \in A, x \sim y \}$ and use $\mathbb{E}_d = E(\bbZ^d)$  as a shorthand notation. Given two  sets $A,\,B\subseteq \bbZ^d$ we denote by $d(A,B) = \inf\{|x-y|_\infty\,:\,x\in A,\,y\in B\}$ the Euclidean distance between two sets, and define the diameter of $A$ by $\mathrm{diam}(A) = \sup\{|x-y|_\infty \, :\, x,y \in A \}$.
\medskip

We now introduce further notation concerning Bernoulli bond percolation. 
We define $\Omega_{\mathrm{bond}}=\{0,1\}^{\bbE_d}$ and endow it with the canonical $\sigma$-algebra $\cF_{\mathrm{bond}}$. We consider the product measure $\bbQ_p$ on $(\Omega_{\mathrm{bond}}, \cF_{\mathrm{bond}})$ such that the canonical coordinates $(\omega_e)_{e \in \bbE_d}$ are independent and for each $e \in \bbE_d$, $\omega_e$ follows a Bernoulli distribution with parameter $p \in [0,1]$, that is $\bbQ_p[\omega_e=1]=1-\bbQ_p[\omega_e=0]=p$. Edges $e$ with $\omega_e=1$ are called \emph{open}. We say a subset of $\bbE_d$ is \emph{open} if all its elements are open. 
We let $\cC$ be the set of open edges and write $x \stackrel{\cC}{\leftrightarrow} y$ if there exists an open path connecting $x$ and $y$. 
For $A\subseteq \bbE_d$, we define $\cF_A=\sigma(\{\omega_e:e\in A\})$.  It is well-known (see, e.g.,~\cite{grimmett1999}) that there exists $p_c\in (0,1)$ such that for $p>p_c$ there is a measurable subset $\widetilde{\Omega}_0 \subseteq \Omega_{\mathrm{bond}}$ with $\bbQ_p[\widetilde{\Omega}_0]=1$ that ensures a unique infinite open cluster for every $\omega \in \widetilde{\Omega}_0$. In what follows, we restrict our attention to $\omega \in \widetilde{\Omega}_0$. We denote by $\cC_\infty=\cC_\infty(\omega)$ the unique infinite cluster. Note that $\cC_\infty$ is viewed as a subgraph of $\bbZ^d$ with edges in $\mathbb{E}_{\mathcal{C}_\infty}  \stackrel{\mathrm{def}}{=} \{\{x,y\} \, : \, \omega_{\{x,y\}} =1, x,y \in \mathcal{C}_\infty \}$. For $x\in \cC_\infty$, we define the weights
\begin{equation}
\label{eq:Weights-given-mu}
    \omega_{x} = \sum_{y \in \mathbb{Z}^d :y\sim x} \omega_{\{x,y\}}, 
\end{equation}
and extend $(\omega_e)_{e\in \mathbb{E}_{\mathcal{C}_\infty}}$ and $(\omega_x)_{x\in \cC_\infty}$ to measures on $\mathbb{E}_{\mathcal{C}_\infty}$ and $\cC_\infty$, respectively. 
We will also need the group of space-shifts on $\Omega_{\mathrm{bond}}$:
\begin{equation}
    \tau_x: \Omega_{\mathrm{bond}} \to \Omega_{\mathrm{bond}}, \quad (\omega_e)_{e\in \bbE_d} \mapsto (\omega_{x+e})_{e\in \bbE_d} ,\quad \text{for $x\in\bbZ^d$}
\end{equation}
(with $x+e=\{x+y,x+z\}$ for $e=\{y,z\} \in \bbE_d$). We can assume, without loss of generality, that $\tau_{x}^{-1}(\widetilde{\Omega}_0)= \widetilde{\Omega}_0 $ for all $x \in \mathbb{Z}^d$, and will tacitly work under this assumption throughout the remainder of the Section.
\subsection{Random walks on \texorpdfstring{$\cC_\infty$}{C\_infinity} and some potential theory} 

We write $\Gamma(\mathbb{R}_{\geq 0},\mathbb{Z}^d)$ for the space of right-continuous, piecewise constant functions from $\mathbb{R}_{\geq 0} \stackrel{\mathrm{def.}}{=} [0,\infty)$ to $\mathbb{Z}^d$ that have finitely many jumps on any bounded interval. We let $X=(X_t)_{t\geq 0}$ be the canonical process on $\Gamma(\mathbb{R}_{\geq 0},\mathbb{Z}^d)$, and write $\mathrm{range}(X) = \{X_t \, : \, t \in \mathbb{R}_{\geq 0} \}$ for the range of $X$. Under $P_x^\omega$, $x \in \cC_\infty$, $X$ is distributed as a continuous-time, constant-speed simple random walk on $\cC_\infty$ starting from $x$. Formally, $X$ can be viewed as a Markov process with generator 
\begin{equation}
\label{eq:Generator}
\cL_{\cC_\infty} f(x)=\frac{1}{\omega_x}\sum_{y\in \cC_\infty}\omega_{\{x,y\}}(f(y)-f(x)), \quad x\in \cC_\infty,
\end{equation}
for $f:\cC_\infty\to \bbR$, with $\omega_x$ as in~\eqref{eq:Weights-given-mu}, which waits an exponential time of mean one at each vertex and jumps to a uniformly chosen nearest neighbor in $\cC_\infty$. Note that 
\begin{equation}\label{eq:tau-RW}
    P_x^{\tau_z\omega}[X_t=y] = P^\omega_{x+z}[X_t=y+z]. 
\end{equation}
Given a probability measure $\nu$ on $\cC_\infty$, we define for a measurable subset $A$ in $\Gamma(\bbR_{\geq 0}, \bbZ^d)$ the probability measure
\begin{equation}\label{eq:P-RW-average}
    P^\omega_\nu[A] = \sum_{x\in \cC_\infty} \nu_x P_x[A], 
\end{equation}
governing a continuous-time random walk on $\mathcal{C}_\infty$ with starting distribution $\nu$. \medskip

We now introduce a notion from \cite{barlow2004RWpercolation} of local regularity which is captured by an event $\cL(Q)$ defined in terms of a cube $Q$ of side-length $L$. As the details are not critical here, we only briefly recall the main elements of the definition in the following. The definition of $\cL(Q)$ is related to the notion of a \emph{very good} box. We refer to \cite[Section~2]{barlow2004RWpercolation} and in particular the equation above Theorem~2.23 therein for the precise definition. Roughly speaking, a box is very good if each of its sub-boxes whose radius exceeds a prescribed regularity scale satisfies both volume regularity and a weak Poincar\'{e} inequality (see Definition~1.7 in \cite{barlow2004RWpercolation} for a precise definition). The event $\cL(Q)$ encapsulates a set of local regularity conditions, which ensures that for $\omega \in \cL(Q)$ all boxes $B(x,\frac{3}{2}R)$ in a slightly enlarged version of $Q$, with $cL^{1/11}\leq R \leq L$, are very good. Furthermore, for any two vertices in $Q$ that are sufficiently far apart, there exists a chain of very good boxes of constant order radius connecting them. 

\begin{lemma}[\cite{barlow2004RWpercolation}, Section 2, Proposition~6.1, Theorem~1]\label{lem:Barlow-HK}
    For $x\in \bbZ^d$, let 
    \begin{equation}\label{eq:S_x}
        S_x = \sup\{L\in \bbN\,:\, \text{$\cL(Q)$ fails for some cube $Q$ with side-length $L$ containing $x$}\}
    \end{equation}
(with the convention that $\sup \varnothing = 0$). \smallskip

    The following properties hold. 
    \begin{equation}\label{eq:LQ-measurable}
        \begin{minipage}{0.8\textwidth}
            For $Q=y+\{0,\dots,L-1\}^d$ with some $y\in \bbZ^d$ and $L\in \bbN$, one has
        \[
        \cL(Q) \in \cF_{E(y+\{-L,\dots,2L-1\}^d)}. 
        \]
        \end{minipage}
    \end{equation}
    \begin{equation}\label{eq:Barlow-HK}
        \begin{minipage}{0.9\textwidth}
            There exists a measurable subset $\Omega_0\subseteq \widetilde{\Omega}_0$ with $\bbQ_p[\Omega_0]=1$ and $\Omega_0=\tau_z^{-1}(\Omega_0)$ for all $z\in \bbZ^d$ such that for $\omega \in \Omega_0$, $S_x(\omega)<\infty$ for all $x\in\bbZ^d$ and for $x,y \in \cC_\infty(\omega)$, $t\geq S_x(\omega) \vee c_0|x-y|_\infty$,  
        \begin{equation*}
            c_{1} t^{-d/2} \exp\Big(-c_2 \frac{|x-y|_\infty^2}{t}\Big) \leq \frac{P_x^\omega[X_t=y]}{\omega_y} \leq c_{3} t^{-d/2} \exp\Big(-c_{4}\frac{|x-y|_\infty^2}{t}\Big).
        \end{equation*}
        \end{minipage}
    \end{equation}
    \begin{equation}\label{eq:S_x-decay}
        \begin{minipage}{0.9\textwidth}
            For $x\in \bbZ^d$ and $n \geq 1$, 
        \begin{equation*}
            \bbQ_p[S_x \geq n] \leq c_5 \exp(-c_{6} n^{\eta_{\mathrm{hk}}}), \text{ where }  \eta_{\mathrm{hk}}= \frac{d-1}{11(d+1)(d+2)}.
        \end{equation*}
        \end{minipage}
    \end{equation}
\end{lemma}

\begin{proof}
    These results are from \cite{barlow2004RWpercolation}. We give some details and specify the exact locations below for the reader's convenience. 
    
    To see \eqref{eq:LQ-measurable}, one first notes, by a careful inspection of the constructions in \cite[Section~2]{barlow2004RWpercolation}, that the defining events $H(Q,\alpha_2)$ and $D(Q,\alpha_2)$ of $\cL(Q)$ (see above \cite[Theroem~2.23]{barlow2004RWpercolation}) are $\cF_{E(Q^+)}$-measurable. This fact was also used in the proof of Lemma~2.21 in the same reference. Therefore, by the definition of $\cL(Q,m,x_0,x_1)$ in \cite[(2.31)]{barlow2004RWpercolation}, the remaining event entering the definition of $\cL(Q)$ is $\cF_{E(y+\{-L,\dots,2L-1\}^d)}$-measurable. 
    
    We turn to \eqref{eq:Barlow-HK} and \eqref{eq:S_x-decay}. It is clear by~\eqref{eq:tau-RW} that one can choose $\tau_z^{-1}(\Omega_0)=\Omega_0$. As we can see from the first paragraph in the proof of \cite[Proposition~6.1]{barlow2004RWpercolation}, the random variables $(S_x)_{x\in \bbZ^d}$ in the statement of Theorem~1 therein are exactly the explicitly defined random variables $(N_x)_{x\in \bbZ^d}$ in \cite[Lemma~2.24]{barlow2004RWpercolation}, which are recalled here in \eqref{eq:S_x}.  Hence, the statement \eqref{eq:Barlow-HK} follows by \cite[Theorem~1]{barlow2004RWpercolation} (with $|\cdot |_1$ replaced by $c|\cdot |_\infty$), and the statement \eqref{eq:S_x-decay} follows by \cite[Lemma~2.24]{barlow2004RWpercolation}. Finally, note that the exponents $\beta$ and $\alpha_2$ appearing in \cite[(2.34)]{barlow2004RWpercolation} are defined in (2.2) and preceding Theorem~2.23 therein, respectively, which leads to the explicit formula of $\eta_{\mathrm{hk}}$ in \eqref{eq:S_x-decay}. 
\end{proof}

In the following, we will always assume that $\omega \in \Omega_0$ (see~\eqref{eq:Barlow-HK}). 
We define for $x,y \in \bbZ^d$ the Green function of the random walk on $\cC_\infty$ (extended by zero to $\bbZ^d$) by
\begin{equation}\label{eq:Green-F}
g^\omega(x,y) = 
\left\{
\begin{aligned}
&\frac{1}{\omega_y}E_x^\omega \left[\int_0^\infty \IND_{\{X_t = y \}} \mathrm{d}t \right] = \int_0^\infty \frac{1}{\omega_y}P_x^\omega[X_t=y] \mathrm{d}t\,, \quad &&\text{if $x,y\in \cC_\infty$,} \\
&0\,, &&\text{otherwise,}
\end{aligned}\right. 
\end{equation} 
which is finite for $\omega \in \Omega_0$ by~\eqref{eq:Barlow-HK} and symmetric. In particular, $X$ is transient on $\cC_\infty$. We also have by \eqref{eq:tau-RW} and $\Omega_0=\tau_x^{-1}(\Omega_0)$ that 
\begin{equation}\label{eq:tau-Green}
    g^{\tau_z\omega}(x,y) = g^\omega(x+z,y+z). 
\end{equation} 
We turn to consequences of \eqref{eq:Barlow-HK} for the Green function. 

\begin{lemma}[\cite{barlow2009}, (6.17), Proposition~6.2]
    \label{lem:Green-function-estimate} We have the following on-diagonal and off-diagonal estimates involving $(S_x)_{x\in \bbZ^d}$ (see~\eqref{eq:S_x}) for $\omega \in \Omega_0$: 
    \begin{equation}\label{eq:on-diag-Green}
        \begin{minipage}{0.9\textwidth}
            For all $x\in \cC_\infty$, 
        \begin{equation*}
            g^\omega(x,x) \leq c_7 S_x. 
        \end{equation*}
        \end{minipage}
    \end{equation}
    \begin{equation}\label{eq:off-diag-Green}
        \begin{minipage}{0.9\textwidth}
            For $x,y\in \cC_\infty$ with $|x-y|_\infty \geq c_8(S_x\wedge S_y)$,
        \begin{equation*}
        \frac{c_9}{|x-y|_\infty^{d-2}} \leq g^\omega (x,y) \leq \frac{c_{10}}{|x-y|_\infty^{d-2}}.
        \end{equation*}
        \end{minipage}
    \end{equation}
\end{lemma}
\medskip

We turn to some potential theoretic notions related to the random walk on $\cC_\infty$. 
For a non-empty finite subset $K$ of $\cC_\infty$, we define the equilibrium measure of $K$ on $\cC_\infty$ by
\begin{equation}\label{eq:equilibrium}
  e_K^\omega(x) = \omega_x P_x^\omega [\widetilde{H}_K=\infty]\IND_K(x), \quad \text{for $x\in \cC_\infty$},
\end{equation}
denote its (finite) total mass, the capacity of $K$, by
\begin{equation}\label{eq:capacity}
  \capa^\omega(K) = \sum_{x\in K} e_K^\omega(x), 
\end{equation}
and define the normalized equilibrium measure of $K$ on $\cC_\infty$ by 
\begin{equation}\label{eq:normalized-equilibrium}
    \widetilde{e}_K^\omega(x) = \frac{e_K^\omega(x)}{\capa^\omega(K)}\quad \text{for $x\in \cC_\infty$}. 
\end{equation}
If $\omega\equiv 1$, we drop $\omega$ from the notation, and note that the definitions in~\eqref{eq:equilibrium}-\eqref{eq:normalized-equilibrium} coincide with the corresponding quantities defined in terms of the classical simple random walk on $\bbZ^d$. The same convention also will also be used for other expressions involving $\omega$ associated with the simple random walk (e.g., $g$ represents the Green function on $\bbZ^d$ and $P_x$ represents the law of a simple random walk on $\bbZ^d$, starting from $x \in \mathbb{Z}^d$). 

We record the last-exit decomposition (see Proposition~7.2 of \cite{barlow2017heatkernel})
\begin{equation}\label{eq:FED}
    P^\omega_x [H_K<\infty] = \sum_{y\in K} g^\omega(x,y) e^\omega_K(y) .
\end{equation}
The capacity of a finite subset $K\subseteq \cC_\infty$ can also be characterized by either (see e.g.,~\cite[(7.3)]{barlow2017heatkernel})
\begin{equation}\label{eq:capacity-variational-1}
    \begin{split}
      \capa^\omega(K) &= \max \Big\{\nu[K]: \nu \in [0,\infty)^K, \sum_{y\in K} g^\omega(x,y)\nu_y \leq 1 \text{ for all $x\in K$}\Big\} \\
      &= \min \Big\{\nu[K]: \nu \in [0,\infty)^K, \sum_{y\in K} g^\omega(x,y)\nu_y \geq 1 \text{ for all $x\in K$}\Big\}  ,
    \end{split}
\end{equation}
or (see, e.g.,~\cite[(3.12)]{drewitz2025} and \cite[(2.1)]{teixeira2009interlacement})
\begin{equation}\label{eq:capacity-variational-2}
    \begin{split}
        \capa^\omega(K)&= 
        (\inf \{\sum_{x,y\in K}\nu_x g^\omega(x,y)\nu_y:\, \nu \text{ probability measure on $K$}\})^{-1}\\
    &=\inf\{\cE^\omega(f,f):\,f \text{ finitely supported and } f\geq 1 \text{ on $K$}\},
    \end{split}
\end{equation}
where 
\begin{equation}
        \cE^\omega(f,f) = \frac{1}{2} \sum_{x,y\in \cC_\infty} \omega_{\{x,y\}} (f(y)-f(x))^2 \quad f: \bbZ^d \to \bbR.
    \end{equation}
It follows by picking $\nu\equiv (\min_{x\in K}\sum_{y\in K} g^\omega(x,y))^{-1}$ or $\nu\equiv (\max_{x\in K}\sum_{y\in K} g^\omega(x,y))^{-1}$ in the maximum and minimum characterizations of the capacity in~\eqref{eq:capacity-variational-1} respectively that 
\begin{equation}\label{eq:capacity-bounds}
    \frac{|K|}{\max_{x\in K} \sum_{y\in K} g^\omega(x,y)} \leq \capa^\omega (K) \leq  \frac{|K|}{\min_{x\in K} \sum_{y\in K} g^\omega(x,y)}. 
\end{equation}

\begin{lemma}\label{lem:monotone-capacity}
    Let $K$ be a finite subset of $\mathbb{Z}^d, d\geq 3$. Then for $\omega \in \Omega_0$, 
    \begin{equation}\label{eq:capacity-monotone}
        \capa^\omega(K \cap \cC_\infty) \leq \capa(K).
    \end{equation}
\end{lemma}

\begin{proof} 
    Due to the variational characterization~\eqref{eq:capacity-variational-2}, 
    \begin{equation}
        \capa^\omega(K\cap \cC_\infty) = \inf\{\cE^\omega(f,f):\,f \text{ finitely supported and } f\geq 1 \text{ on $K\cap \cC_\infty$}\}.
    \end{equation}
    For any finitely supported function $f \geq 1$ on $K$ (in particular, $f\geq 1$ on $K\cap \cC_\infty$),
    \begin{equation}
        \capa^\omega(K\cap \cC_\infty) \leq \cE^\omega(f,f) \stackrel{\omega_{\{x,y\}}\leq \IND_{\{x\sim y\}}}{\leq} \frac{1}{2} \sum_{x\sim y\in \bbZ^d} (f(y)-f(x))^2.
    \end{equation}
    Therefore by taking infimum and \eqref{eq:capacity-variational-2} with $\omega\equiv 1$, we obtain $\capa^\omega(K\cap \cC_\infty) \leq \capa(K)$. 
\end{proof}

\subsection{Gaussian free field on \texorpdfstring{$\cC_\infty$}{C\_infinite}}

For $\omega \in \Omega_0$, we recall from the paragraph around~\eqref{eq:GFF-cov} the canonical coordinates $(\varphi_x)_{x\in \bbZ^d}$ and the law $\bbP_\omega^G$ on $(\bbR^{\bbZ^d},\cF^G)$ that governs the GFF on $\cC_\infty$. We also recall the level sets $E_\omega^{\geq \alpha}$ for $\alpha\in \bbR$ in~\eqref{eq:level-set}. Note that the definition depends on both sources of randomness $\omega$ and $\varphi$, and thus we sometimes write $E_\omega^{\geq\alpha}(\varphi)$. 

\begin{lemma}\label{lem:GFF-capa}
    Let $\omega \in \Omega_0$ and $K$ be a finite subset of $\cC_\infty$. Then for $\alpha >0 $
    \begin{equation}\label{eq:GFF-capa}
        \bbP^G_\omega[K\subseteq E^{\geq \alpha}] \leq \exp\Big\{-\tfrac{\alpha^2}{2} \capa^\omega(K)\Big\}. 
    \end{equation}
\end{lemma}

\begin{proof}
    Note that 
    \begin{equation}
        \var^G_\omega \Big[ \sum_{x\in K} \varphi_x e^\omega_K(x) \Big] = \sum_{x,y\in K} g^\omega(x,y) e_K^\omega(x) e_K^\omega(y) \overset{\eqref{eq:FED}}{=} \capa^\omega (K).
    \end{equation}
    We then have 
    \begin{equation}
        \bbP^G_\omega [K\subseteq E^{\geq \alpha}] \leq \bbP^G_\omega \Big[ \sum_{x\in K} \varphi_x e^\omega_K(x) \geq \alpha \capa^\omega (K) \Big] \leq\exp\Big\{-\tfrac{\alpha^2}{2} \capa^\omega(K)\Big\},
   \end{equation}
    where we used that for a standard Gaussian random variable $\psi$, it holds $\mathsf{P}[\psi \geq x] \leq \min_{\lambda > 0} e^{-\lambda x} \mathsf{E}[e^{\lambda \psi}]=\min_{\lambda > 0} e^{-\lambda x+x^2/2} = \exp(-x^2/2)$ for $x>0$, as the minimum is attained at $\lambda =x$.
\end{proof}

We consider the group of space-shifts on $\bbR^{\bbZ^d}$:
\begin{equation}
\label{eq:Def-t-x}
    t_x: \bbR^{\bbZ^d} \to \bbR^{\bbZ^d}, \quad \varphi_{\cdot} \mapsto \varphi_{\cdot +x} \quad \text{for $x\in \bbZ^d$}. 
\end{equation}
For $\omega \in \Omega_0$, it follows from Gaussianity, $\tau_x^{-1}(\Omega_0)=\Omega_0$ (see \eqref{eq:Barlow-HK}), and \eqref{eq:tau-Green} that for all bounded measurable functions $f:\bbR^{\bbZ^d} \to \bbR$, 
\begin{equation}\label{eq:shift-t}
    \text{$\bbE^G_\omega[f(t_x\varphi)] = \bbE^G_{\tau_x\omega}[f(\varphi)]$ for all $x\in \bbZ^d$}.
\end{equation}

\begin{lemma}\label{lem:quenched-mixing}
    It holds for $\bbQ_p$-a.e.~$\omega$ that if $H,K$ are finite non-empty subsets of $\bbZ^d$ and $A\in \sigma((\varphi_y)_{y\in H})$, $B\in \sigma((\varphi_y)_{y\in K})$, then for all $x\in \bbZ^d$ such that $d(H,x+K)\geq \sup_{y\in H} S_y$ it holds
    \begin{equation}\label{eq:quenched-mixing}
    \Big| \bbP^G_\omega[A\cap t_x^{-1}(B)] -   \bbP^G_\omega[A]  \bbP^G_\omega[t_x^{-1}(B)]\Big| \leq c \frac{\sqrt{\capa(H)\capa(K)}}{ d(H,x+K)^{d-2}}.
   \end{equation}
   In particular, for $\bbQ_p$-a.e.~$\omega$
   \begin{equation}\label{eq:quenched-mixing2}
   \lim_{x\to\infty} \Big| \bbP^G_\omega[A\cap t_x^{-1}(B)] -   \bbP^G_\omega[A]  \bbP^G_\omega[t_x^{-1}(B)]\Big| = 0.
   \end{equation}
\end{lemma}

\begin{proof} As a first observation, if either $H\cap \cC_\infty = \varnothing$ or $(x+K)\cap \cC_\infty = \varnothing$, then one of $\bbP^G_\omega[A]$ or $\bbP^G_\omega[t_x^{-1}(B)]$ belongs to $\{0,1\}$ and thus \eqref{eq:quenched-mixing} follows. We therefore assume without loss of generality that $\overline{H} = H\cap \cC_\infty$ and $\overline{K}_x = (x+K)\cap \cC_\infty$ are both non-empty.

It holds from the proof of Theorem 1.1 in~\cite{popov2015transition} (see~(3.1)--(3.4) therein) that 
\begin{equation}\label{eq:sup}
    \Big| \bbP^G_\omega[A\cap t_x^{-1}(B)] -   \bbP^G_\omega[A]  \bbP^G_\omega[t_x^{-1}(B)]\Big|\leq \sup_{\substack{X\in \bfH,Y\in \bfK_x\\X\neq0,\,Y\neq 0}} \frac{\bbE^G_\omega[XY]}{\sqrt{\bbE^G_\omega[X^2] \bbE^G_\omega[Y^2]}}\,,
\end{equation}
where $\bfH\subseteq L^2(\bbP^G_\omega)$ is the span of $\{\varphi_y\,:\,y\in \overline{H}\}$ and $\bfK_x\subseteq L^2(\bbP^G_\omega)$ is the span of $\{\varphi_y\,:\,y\in \overline{K}_x\}$, having used that $\bbP^G_\omega$-a.s.\ $\varphi_y = 0$ if $y\notin\cC_\infty$. Then, following the steps  from (3.4) and (3.7) in~\cite{popov2015transition}, the right hand side of \eqref{eq:sup} can be bounded by
\begin{equation}
    \sup_{\alpha,\beta} \frac{\sum_{y\in \overline{H}}\sum_{u\in \overline{K}_x} \alpha_y\beta_u\, g^\omega(y,u)}{\sqrt{\sum_{y,z\in \overline{H}} \alpha_y\alpha_z\,g^\omega(y,z) \sum_{u,w\in \overline{K}_x} \beta_u\beta_w\, g^\omega(u,w)}},
\end{equation}
where the sup is taken over $\alpha:\overline{H}\to [0,1]$ and $\beta:\overline{K}_x\to [0,1]$ probabilities on $\overline{H}$ and $\overline{K}_x$ respectively. Note that by~\eqref{eq:capacity-variational-2}
\begin{equation}
   \frac{1}{\capa^\omega(\overline{H})} = \inf_{\alpha} \sum_{y,z\in \overline{H}} \alpha_y\alpha_z\,g^\omega(y,z)\,,\quad \frac{1}{\capa^\omega(\overline{K}_x)} = \inf_{\beta} \sum_{u,w\in \overline{K}_x} \beta_u\beta_w\, g^\omega(u,w)\,.
\end{equation}
Therefore, for all $x\in \bbZ^d$ such that $d(H,x+K)\geq \sup_{y\in H} S_y$, it follows from~\eqref{eq:off-diag-Green} that this quantity is bounded by
\begin{equation}
    \sup_{y\in \overline{H},u\in \overline{K}_x} g^{\omega}(y,u) \sqrt{\capa^\omega(\overline{H})\capa^\omega(\overline{K}_x)}\leq  \frac{ c_{10}\sqrt{\capa^\omega(\overline{H})\capa^\omega(\overline{K}_x)}}{\inf_{y\in \overline{H},u\in \overline{K}_x} |y-u|_\infty^{d-2}}\,.
\end{equation}
Inequality~\eqref{eq:quenched-mixing} then follows from Lemma~\ref{lem:monotone-capacity}. Finally, \eqref{eq:quenched-mixing2} follows from the fact that $\sup_{y\in H}S_y$ is finite $\bbQ_p$-almost surely.
\end{proof}

Let $\bfP^G(d\omega, d\varphi)=\bbQ_p(d\omega) \bbP^G_\omega(d\varphi)$ be the annealed measure on $\Omega_{\mathrm{bond}} \times \bbR^{\bbZ^d}$ endowed with the product $\sigma$-algebra. We define on $\Omega_{\mathrm{bond}} \times \bbR^{\bbZ^d}$ the translation operator 
\begin{equation}
    \Theta_x(\omega, \varphi) = (\tau_x \omega, t_x \varphi) \quad \text{for $x\in \bbZ^d$}. 
\end{equation}

\begin{proposition} 
    The annealed measure $\bfP^G$ is invariant and ergodic under $(\Theta_x)_{x\in \bbZ^d}$.
\end{proposition}

\begin{proof}
    We first show that $\bfP^G$ is invariant under $\Theta_x$ for $x\in \bbZ^d$: for all $A\in \cF_{\mathrm{bond}} \otimes \cF^G$, 
    \begin{equation}
    \begin{split}
        \bfP^G[\Theta_x^{-1}(A)] & = \int \Big(\int \IND_A(\tau_x\omega,t_x\varphi) \bbP^G_\omega (d\varphi) \Big) \bbQ_p(d\omega)\\
        &\stackrel{\eqref{eq:shift-t}}{=} \int_{\tau_x^{-1}(\Omega_0)} \Big(\int \IND_A(\tau_x \omega, \varphi) \bbP^G_{\tau_x\omega} (d\varphi) \Big) \bbQ_p(d\omega) 
       \\
       &  =\int_{\Omega_0} \Big(\int \IND_A(\omega, \varphi) \bbP^G_{\omega} (d\varphi) \Big) \bbQ_p(d\omega)= \bfP^G[A].  
        \end{split}
    \end{equation}
    We turn to the ergodicity. We first consider $A=A^1\times A^2$ and $B=B^1\times B^2$ with $A^1,B^1\in \cF_{\mathrm{bond}}$ and $A^2,B^2 \in \cF^G$, all depending only on finitely many coordinates. Then 
    \begin{equation}
    \begin{split}
        &\big|\bfP[ A \cap \Theta_x^{-1}(B)] - \bfP[A] \bfP[B]\big| \\
        &\leq \Big|
        \int \IND_{A^1}(\omega) \IND_{B^1}(\tau_x\omega) \left(\bbP_\omega^G[A^2\cap t_x^{-1}(B^2)]-\bbP_\omega^G[A^2]\bbP^G_\omega[t_x^{-1}(B^2)]  \right) \bbQ_p(d\omega)
        \Big| \\
        &\quad +
        \Big|
        \int \IND_{A^1}(\omega) \bbP_\omega^G[A^2] \cdot \IND_{B^1}(\tau_x\omega) \bbP^G_{\tau_x\omega}[B^2] \bbQ_p(d\omega) - \bfP[A] \bfP[B]
        \Big| \\
        &\to 0 \quad \text{as $|x|\to\infty$}, 
            \end{split}
    \end{equation}
    by applying the dominated convergence theorem and~\eqref{eq:quenched-mixing} for the first term, and using the mixing and translation invariance property of an i.i.d.~field of Bernoulli random variables indexed by $\bbZ^d$ for the second term. General events $A$ and $B$ can be treated via an approximation argument.
\end{proof}

The ergodicity of the annealed measure implies that the quenched law of any invariant event has $\bbQ_p$-a.s.~a zero-one law. 

\begin{lemma}\label{lem:0-1-law-GFF}
    If $A$ is invariant under $(\Theta_x)_{x\in \bbZ^d}$, then for $\bbQ_p$-a.e.~$\omega$, $\bbP_\omega^G[A]\in \{0,1\}$. 
    
    In particular, for any $\alpha\in \bbR$, there exists a measurable set $\Omega_1(\alpha)\subseteq \Omega_0$ with $\bbQ_p[\Omega_1(\alpha)]=1$ and $\delta_\alpha \in \{0,1\}$ such that for all $\omega \in \Omega_1(\alpha)$,  
    \begin{equation}
        \bbP^G_\omega\big[\text{$E_\omega^{\geq \alpha}$ contains an infinite cluster\,}\big] =\delta_\alpha. 
    \end{equation}
\end{lemma}

\begin{proof}
    It follows by ergodicity that $\bfP^G[A]\in \{0,1\}$. Since $\bbP^G_\omega[A]\in [0,1]$ and its $\bbQ_p$-integral is either $0$ or $1$, its value must be $0$ or $1$,~$\bbQ_p$-a.s. 

    The second claim follows by noting that the event \{$E_\omega^{\geq \alpha}$ contains an infinite cluster\} is invariant with respect to $(\Theta_x(\omega, \varphi))_{x\in \bbZ^d}$: Indeed, as
    \begin{equation}
    \begin{split}
        E_{\tau_x\omega}^{\geq \alpha}&(t_x\varphi) = \{y\in \bbZ^d: y\in \cC_\infty(\tau_x\omega), (t_x \varphi)_y \geq \alpha\} \\
        &= \{y\in \bbZ^d:y+x \in \cC_\infty(\omega),\varphi_{y+x} \geq \alpha\} = E_\omega^{\geq \alpha}(\varphi)-x, 
        \end{split}
    \end{equation}
    it is clear that both random subsets contain infinite clusters simultaneously. 
\end{proof}

We briefly compare the Gaussian free field on a typical realization of $\mathcal{C}_\infty$ to that on $\mathbb{Z}^d$, $d \geq 3$, and highlight a potential obstruction in proving the phase transition of $E^{\geq \alpha}_\omega$.
\begin{remark}\label{rmk:difficulty-1}
    By the monotonicity of the capacity (Lemma~\ref{lem:monotone-capacity}) and \eqref{eq:FED}, we see that for $x\in \cC_\infty$,
$g^\omega(x,x) = \frac{1}{\capa^\omega(\{x\})} \geq g(x,x)$.
The increase of the variances when compared to the field on $\mathbb{Z}^d$ can heuristically be related to a general increase of the maximum of the field over large sets: In fact, by \cite[Corollary 1.7]{abe2015effective}, we have $\bbQ_p$-a.s.~for $N$ large enough
\begin{equation}
    c \log N \leq \max_{x\in [0,N)^d \cap \cC_\infty} \varphi_x \leq C \log N,
\end{equation}
whereas the order of the maximum over  $[0,N)^d \cap \mathbb{Z}^d$ is $\sqrt{\log N}$ for the field on $\bbZ^d$. One source of this different behavior is the presence of ``tubes'' of length $O(\log N)$ in $[0,N)^d \cap \mathcal{C}_\infty$. This suggests a potentially competing effect when considering the set $E^{\geq \alpha}_\omega$ for $\alpha \geq 0$. On one hand, the maximum over a mesoscopic subset of $\mathcal{C}_\infty$ is more likely to exceed the value $\alpha$ than for the corresponding field on $\mathbb{Z}^d$. On the other hand points of a high maximum are associated with ``narrow pathways'' in $\mathcal{C}_\infty$ (where the variance is large), which do not necessarily contribute to the connectivity of $E^{\geq \alpha}_\omega$. We also refer to \cite[Remark~6.4~(3)]{chiarini2025hardwall} for the conjectured effect of this phenomenon for the field conditioned on staying non-negative over a macroscopic set.

\end{remark}

\subsection{Random interlacements on \texorpdfstring{$\cC_\infty$}{C\_infinity}}
\label{subsec:RI}

We consider a configuration $\omega\in \Omega_0$ (with $\Omega_0$ as in Lemma~\ref{lem:Barlow-HK}) such that  $\cC_\infty(\omega)$ is a unique infinite cluster on which the simple random walk is transient. By~\cite[Remark 1.4]{sznitman2010}, since $\mathcal{C}_\infty$ as a locally finite, connected, transient graph, one can define the random interlacements $(\mathcal{I}^u_\omega)_{u \geq 0}$ on $\mathcal{C}_\infty$ under a probability measure $\mathbb{P}^I_\omega$ for each fixed $\omega \in \Omega_0$ (see also \cite{teixeira2009interlacement} for the construction). \medskip

The full construction of random interlacements on $\cC_\infty$ involves a Poisson point process under $\bbP_\omega^I$ with intensity measure $u\cdot \nu_\omega$, where $\nu_\omega$ is a $\sigma$-finite measure on the space of equivalence classes of doubly infinite trajectories in $\cC_\infty$, see \cite[Theorem 2.1]{teixeira2009interlacement}. The probability measure $\bbP_\omega^I$ is formally defined on a space of point measures $\Omega_\omega^{\mathrm{pm}}$ of doubly-infinite trajectories up to time-shift, each labelled with a non-negative real number. The union of the trajectories which are contained in the support of a sample from this Poisson point process is then denoted by $\cI^u_\omega (\subseteq \mathcal{C}_\infty)$. We let $\cV^u_\omega= \cC_\infty \setminus \cI^u_\omega$ stand for the corresponding vacant set of $\mathcal{I}^\omega_u$. The law of the random subset $\cV^u_\omega$ satisfies
\begin{equation}\label{eq:law-Vu}
    \bbP^I_\omega[K\subseteq \cV^u] = \bbP^I_\omega[\cI^u\cap K=\varnothing] = \exp(-u\capa^\omega(K)), \quad \text{for finite $K\subseteq \cC_\infty$}, 
\end{equation}
and it is characterized by this formula, see \cite[Remark~2.3~(1)]{teixeira2009interlacement}. Although it will not be required in the following, we briefly point out that one may naturally embed $\Omega^{\mathrm{pm}}_\omega$ it into the space $\Omega^{\mathrm{pm}}$ (taking $\mathbb{Z}^d$ as a base graph) introduced in \cite[Eq.~(1.16)]{sznitman2010},
\begin{equation}
    \Omega^{\mathrm{pm}} = \left \{ \eta = \sum_{i\geq 0} \delta_{(w_i^*,u_i)} : 
    \begin{minipage}{0.6\textwidth}
        each $w_i^*$ is a doubly-infinite trajectory on $\bbZ^d$, $u_i \in \bbR_{\geq 0}$,\\
        for any finite $K \subseteq \bbZ^d$ and $u < \infty$, 
        the number of trajectories in $\eta$ hitting $K$ with label $u_i \le u$ is finite
\end{minipage}
    \right\},
\end{equation}
by assigning $\bbP^I_\omega [\Omega^{\mathrm{pm}} \setminus \Omega^{\mathrm{pm}}_\omega]=0$. \medskip

We now state a local construction of random interlacements in a finite subset of $\mathcal{C}_\infty$ with the same law as~\eqref{eq:law-Vu}, whose proof translates directly to our set-up.  We recall that $\omega \in \Omega_0$ is fixed.

\begin{lemma}[\cite{rath2015transition}, Claim 2.2]\label{lem:constructive-RI}
    Let $K$ be a finite subset of $\cC_\infty$, $u > 0$, $N_K$ be a Poisson random variable with parameter $u\cdot \capa^\omega(K)$, and $(X^j)_{j\geq 1}$ be i.i.d.~simple random walks with distribution $P^\omega_{\widetilde{e}_K^\omega}$ (see \eqref{eq:P-RW-average} and~\eqref{eq:normalized-equilibrium} for notation) and independent from $N_K$, under some rich enough probability measure $\mathsf{P}^\omega$. Then $K\cap \bigcup_{j=1}^{N_K} \range(X^j)$ under $\mathsf{P}^\omega$ has the same distribution as $\cI^u \cap K$ under $\bbP^I_\omega$.  
\end{lemma}

The following zero-one law is an application of the recent result~\cite[Theorem 12]{collin2026zeroone}. Note that if the symmetric difference between two subsets of $\mathbb{Z}^d$ is finite, they either both contain an infinite cluster or neither does. This implies that the event that $\mathcal{V}^u_\omega$ contains an infinite cluster is measurable with respect to the information associated with trajectories outside any finite set $K$, and thus belongs to the tail $\sigma$-algebra $\widetilde{\cT}_{\mathrm{RI}}$ (see below~\cite[(38)]{collin2026zeroone} therein). Together with the fact that this event is \emph{decreasing} (see ~\cite[(18) and (22)]{collin2026zeroone}), we have the following.

\begin{lemma}[\cite{collin2026zeroone}, Theorem 12]
    For $u>0$, there exists $\delta_u \in \{0,1\}$ such that for $\bbQ_p$-a.e.~$\omega$, 
    \begin{equation}
        \bbP^I_\omega[\text{$\cV_\omega^u$ contains an infinite cluster}] = \delta_u. 
    \end{equation}
\end{lemma}

As before, we comment on a ``competing effect'' when comparing random interlacements on $\mathbb{Z}^d$ to those on $\mathcal{C}_\infty$ which is an obstruction for the proof of a percolation phase transition of the corresponding vacant set on $\mathcal{C}_\infty$. 
\begin{remark}\label{rmk:difficulty-2}
    Due to the monotonicity of the capacity (Lemma~\ref{lem:monotone-capacity}) and the explicit characterization~\eqref{eq:law-Vu}, a given finite set $K$ in $\cC_\infty$ is more likely to be vacant for $\mathcal{I}^\omega_u$ than the corresponding set in $\mathbb{Z}^d$ for $\mathcal{I}^u$ (the interlacements on $\mathbb{Z}^d$). This heuristically might facilitate the existence of a supercritical phase but hinder the existence of a subcritical phase of $\mathcal{V}^u_\omega$. On the other hand, certain parts of the underlying graph $\mathcal{C}_\infty$ are reduced to long ``tubes'' or ``dead ends'' (see~\cite{abe2015effective}), which may be easier to cut off by $\mathcal{I}^u_\omega$ and reduce the connectedness of $\mathcal{V}^u_\omega$ and suggesting the opposite heuristic.
\end{remark}

\section{Quenched uniform estimates over tree embeddings}

\label{sec:Quenched-Unif-Estimates}

In this Section, we develop the main technical tool of the article in Theorem~\ref{thm:general-proportion}. We define a notion of \emph{approximability} for a non-local spatial field of Bernoulli random variables, which heuristically states that the latter can be approximated by a field of local Bernoulli random variables at every scale (see Definition~\ref{def:approximable}). This is then combined with a renormalization scheme based on tree embeddings into $\bbZ^d$ (see Definition~\ref{def:Tree-embeddings}) taken from~\cite{rath2015transition} (see also~\cite{sznitman2012decoupling}). Our main result shows that for an approximable field of Bernoulli random variables with a sufficiently small parameter, 
the number of vertices with value $1$ at the lowest scale of the tree embedding 
can be \textit{uniformly} controlled over all such embeddings. \smallskip 

This result will be applied in Sections~\ref{sec:Subcritical-Phase} and~\ref{sec:Supercritical-Phase} to control the behavior of the Green function on boxes of a constant size attached to the leaves of the tree, and, in Section~\ref{sec:Supercritical-Phase}, to further guarantee certain advantageous geometric properties of a ``renormalized'' version of the cluster.
\smallskip

We first introduce some notation and recall the static renormalization scheme we will use (see Lemma~\ref{lem:renormalization} below). A variant of this scheme was originally introduced in \cite{sznitman2010} (but with super-exponentially growing scales), and the version we use goes back to~\cite{sznitman2012decoupling} for more general graphs (here, we use a version from~\cite{rath2015transition} tailored to $\mathbb{Z}^d$). Throughout this Section, we let $d \geq 2$.

\begin{definition}[see \cite{rath2015transition}, Section 3]
\label{def:Tree-embeddings}
    Let $L_0 \geq 1$, $\ell_0\geq 100$, $n\geq 0$ be integers, and define 
\begin{equation}\label{eq:L_n}
    L_n=  L_0 \ell_0^n \quad \text{and} \quad \bbL_n = L_n \mathbb{Z}^d. 
\end{equation}
\begin{equation}
    \begin{minipage}{0.9\textwidth}
        Let $T_{(n)}=\{1,2\}^n$ (with $T_{(0)}=\{\varnothing\}$) and 
    \[
    T_n = \bigcup_{k=0}^n T_{(k)}. 
    \]
    For $0\leq k < n$ and $m=(\xi_1,\dots,\xi_k)\in T_{(k)}$, we denote by $m_1=(\xi_1,\dots,\xi_k,1)$ and $m_2=(\xi_1,\dots,\xi_k,2)$ the two children of $m$ in $T_{(k+1)}$. 
    \end{minipage}
\end{equation}
\begin{equation}\label{eq:T-m-1-2}
    \begin{minipage}{0.9\textwidth}
        For $n\geq 0$, we say that $\cT:T_n \to \mathbb{Z}^d$ is a \emph{proper tree embedding} of $T_n$ at root $x\in \bbL_n$ with depth $n$ if $\cT(\varnothing)=x$, and if for all $0\leq k \leq n$, $m\in T_{(k)}$, $\mathcal{T}(m)\in \bbL_{n-k}$, and if for all $0\leq k < n$, $m\in T_{(k)}$  it holds that  
    \begin{equation*}
        |\cT(m_1)-\cT(m)|_\infty = L_{n-k}, \quad |\cT(m_2)-\cT(m)|_\infty = 2L_{n-k}.
    \end{equation*}
    \end{minipage}
\end{equation}
\begin{equation}
    \begin{minipage}{0.9\textwidth}
        We denote by $\Lambda_{n,x}$ the set of all proper tree embeddings of $T_n$ rooted at $x\in \bbL_n$.
    \end{minipage}
\end{equation}
\begin{equation}\label{eq:sub-division-tree}
    \begin{minipage}{0.9\textwidth}
        For $0\leq k \leq n$ and $m=(\xi_1,\dots,\xi_n)\in T_{(n)}$, we define $m|_k = (\xi_1,\dots,\xi_k)\in T_{(k)}$. We denote the lexicographic distance of $m,m'\in T_{(n)}$ by 
        \begin{equation*}
            \rho(m,m') = \min\{k\geq 0:m|_{n-k} = m'|_{n-k}\},
        \end{equation*}
        and for $m\in T_{(n)}$ and $1\leq k \leq n$ we define 
        \begin{equation*}
            T_{(n)}^{m,k} = \{m' \in T_{(n)}:\rho(m,m')=k\} \quad (\text{with }|T_{(n)}^{m,k}|=2^{k-1}) 
        \end{equation*}
        (see \cite[Figure 1]{rath2015transition} for a visualization). 
        \end{minipage}
\end{equation}
\begin{equation}\label{eq:cT_i}
    \begin{minipage}{0.9\textwidth}
        For $\cT\in \Lambda_{n,x}$, we denote by $\cT_{i}$ the restriction of $\cT$ on $\{m\in T_n \,:\, m|_1=i\}$ for $i\in \{1,2\}$. We note that each of $\cT_i$ can be viewed as a proper tree embedding at root $\cT(i)$ with depth $n-1$. 
    \end{minipage}
\end{equation}
\end{definition}

\begin{lemma}[\cite{rath2015transition}, Lemma~3.2]
    Let $n\geq 0$ and $x\in \bbL_n$, then 
    \begin{equation}\label{eq:Lambda_n-x-number}
        |\Lambda_{n,x}| = \cC_{d,\ell_0}^{2^n-1}, \quad \text{where }\cC_{d,\ell_0} = \big((2\ell_0+1)^d-(2\ell_0-1)^d\big)\big((4\ell_0+1)^d-(4\ell_0-1)^d\big). 
    \end{equation}
\end{lemma}

\begin{proof}
    The proof is an adaptation of \cite[Lemma~3.2]{rath2015transition}, and we briefly describe the necessary modifications. For $m\in T_{(k)}$, $0\leq k <n$, the construction in \eqref{eq:T-m-1-2} shows that the number of choices for $\cT(m_1)$ is 
    \begin{equation}
        |S(\cT(m), L_{n-k}) \cap \bbL_{n-k-1}| =  \big((2\ell_0+1)^d-(2\ell_0-1)^d\big),
       \end{equation}
    and similarly the number of choices for $\cT(m_2)$ is 
    \begin{equation}
    |S(\cT(m), 2L_{n-k}) \cap \bbL_{n-k-1}| =  \big((4\ell_0+1)^d-(4\ell_0-1)^d\big).
       \end{equation}
    The result then follows as in the proof of~\cite[Lemma 3.2]{rath2015transition}. 
\end{proof}

We also record the following results for later use. By slight abuse of notation, we identify a $*$-connected path $\gamma : \{1,...,n\} \rightarrow \bbZ^d$ with its image in $\bbZ^d$.

\begin{lemma}[\cite{rath2015transition}, Lemma 3.3]\label{lem:renormalization}
    Let $d\geq 2$, $n\geq 1$, $x\in \bbL_n$. If a $*$-connected path $\gamma$ satisfies $\gamma\cap S(x,L_n-1) \neq \varnothing$ and $\gamma\cap S(x,2L_n) \neq \varnothing$, then there exists $\cT \in \Lambda_{n,x}$ such that 
    \begin{equation}
        \gamma \cap S(\cT(m), L_0-1)\neq \varnothing \quad \text{for all $m\in T_{(n)}$}. 
    \end{equation}
\end{lemma}

\begin{lemma}[\cite{rath2015transition}, Lemma~3.4]\label{lem:tree-geometry}
    Let $n\geq 1$, $x\in\bbL_n$, $\cT\in \Lambda_{n,x}$, $\ell_0 \geq 100$. For $m\in T_{(n)}$, $1\leq k \leq n$ and $m'\in T_{(n)}^{m,k}$, 
    \begin{equation}\label{eq:tree-distance}
        |\cT(m)-\cT(m')|_\infty \geq 90 L_{k-1} . 
    \end{equation}
\end{lemma}

\begin{proof}
    The proof is adapted from \cite[Lemma~3.4]{rath2015transition}, and we sketch the changes. Let $m''=m|_{n-k}=m'|_{n-k}$ and assume without loss of generality that $m|_{n-k+1}=m''_1$, $m'|_{n-k+1}=m''_2$. Then $|T(m_1'')-T(m_2'')| \geq L_k=\ell_0L_{k-1}$ and 
    \begin{equation}
        |T(m''_1)-\cT(m)|_\infty \leq \sum_{j=1}^{k-1}2L_j \leq 2L_{k-1}\frac{\ell_0}{\ell_0-1}. 
        \end{equation}
    Similarly, $|T(m''_2)-\cT(m')|_\infty \leq2L_{k-1}\frac{\ell_0}{\ell_0-1}$. Therefore, 
    \begin{equation}
    |\cT(m)-\cT(m')|_\infty \geq \ell_0L_{k-1} - 4L_{k-1}\frac{\ell_0}{\ell_0-1} = \ell_0 L_{k-1} \cdot \frac{\ell_0-5}{\ell_0-1} \geq \frac{9}{10} \ell_0 L_{k-1}. 
        \end{equation}
    The result follows as $\ell_0 \geq 100$.
\end{proof}

We now introduce a generic family of local Bernoulli random variables attached to every scale of the renormalized lattices $\bbL_n$, $n \geq 0$, which is assumed to admit a \emph{stablility between consecutive scales}. This notion is made precise in the assumption below.

\begin{assumption}\label{ass:local-rv}
    For each $n \geq 0$, consider a collection of Bernoulli random variables $(\zeta^{(n)}_x)_{x\in \bbL_0}$ defined on the probability space $(\Omega_{\mathrm{bond}},\cF_{\mathrm{bond}}, \bbQ_p)$. We assume that the collections $(\zeta^{(n)}_x)_{x\in \bbL_0}$, $n \geq 0$, satisfy the following properties. 
    \begin{equation}\label{eq:identical-distributed}
        \begin{minipage}{0.9\textwidth}
            (\emph{Identically distributed})\\
            For fixed $n\geq 0$, the random variables $(\zeta^{(n)}_x)_{x\in \bbL_0}$ are identically distributed.
        \end{minipage}
    \end{equation}
    \begin{equation}\label{eq:local}
        \begin{minipage}{0.9\textwidth}
             (\emph{Local}) For each $x\in \bbL_0$, $\zeta^{(n)}_x$ is $\cF_{E(B(x,3L_n))}$-measurable. 
        \end{minipage}
    \end{equation}
    \begin{equation}\label{eq:stable}
        \begin{minipage}{0.9\textwidth}
            (\emph{Stable between scales}) For $n\geq 0$, $x\in \bbL_0$, there exists $E_{n,x}^\zeta\in \cF_{\mathrm{bond}}$ such that $$\zeta^{(n)}_x \IND_{E_{n,x}^\zeta}=\zeta^{(n+1)}_x \IND_{E^\zeta_{n,x}}.$$
        \end{minipage}
    \end{equation}
    \begin{equation}\label{eq:E_n}
        \begin{minipage}{0.9\textwidth}
            Moreover, the events $(E_{n,x}^\zeta)_{x\in \bbL_0}$ are such that 
        \begin{equation*}
            \bbQ_p[(E_{n,x}^\zeta)^c]\leq \varepsilon(L_n) \stackrel{\mathrm{def.}}{=} c_{11} \exp(-c_{12}L_{n}^\beta)\quad  \text{for some $\beta>0$}. 
        \end{equation*}
        \end{minipage}
    \end{equation}
\end{assumption}
\begin{lemma} \label{lem:multi-itera}
    Suppose the collections $(\zeta^{(n)}_x)_{x\in \bbL_0}$, $n \geq 0$, of Bernoulli random variables satisfy Assumption \ref{ass:local-rv} and $\ell_0 \geq 3^{1/\beta} \vee 100$.  Then for $n \geq 1$, and $t \geq 0$,
    \begin{equation}
        \begin{aligned}
          &\sup_{x\in\bbL_n}\sup_{\cT\in \Lambda_{n,x}} \bbE_{\bbQ_p} \Big[ \exp\big( t\sum_{m\in T_{(n)}} \zeta^{(n)}_{\cT(m)} \big)\Big] \\
          & \quad \leq \exp\left\{ \Big(t+\log(e^{-t}\bbQ_p[\zeta^{(0)}_0=0]+ \bbQ_p[\zeta^{(0)}_0=1] + \Delta(L_0))\Big) \cdot 2^n \right\},   
        \end{aligned}
    \end{equation}
    where 
    \begin{equation}\label{eq:Delta-L0}
        \Delta(L_0) = c_{13}\exp(-c_{14} L_0^\beta).
    \end{equation}
\end{lemma}

\begin{proof}
    Recall the events $(E^\zeta_{n,x})_{n\geq 0,x\in \bbL_0}$ in Assumption \ref{ass:local-rv}. For any $n\geq 0$ and any proper tree embedding $\cT$ at some root $x \in \bbL_n$ and depth $n \geq 0$, we let  
    \begin{equation}
        E_{n,\cT} = \bigcap_{m \in T_{(n)}} E^\zeta_{\cT(m), n}. 
    \end{equation}
    We abbreviate $\varepsilon(L_n)$ as $\varepsilon_n$. Note that by \eqref{eq:E_n} 
    \begin{equation}\label{eq:E_n-cT}
        \bbQ_p[(E_{n,\cT})^c] \leq 2^n \varepsilon_n. 
    \end{equation}
    For fixed $n\geq 1$, $x\in \bbL_n$, and $\cT\in \Lambda_{n,x}$, we have 
    \begin{equation}\label{eq:recursion-1}
        \begin{split}
        \EQ &\Big[ \exp\Big( t\sum_{m\in T_{(n)}} \zeta^{(n)}_{\cT(m)} \Big)\Big] \\
        &\stackrel{\eqref{eq:E_n-cT}}{\leq} \EQ\Big[ \exp\Big( t\sum_{m\in T_{(n)}} \zeta^{(n)}_{\cT(m)} \Big) \cdot \IND_{E_{n-1,\cT_1} \cap E_{n-1,\cT_2}}\Big] + e^{t2^n} \cdot 2^n \varepsilon_{n-1} \\
        & \underset{\eqref{eq:stable}}{\overset{\eqref{eq:cT_i}}{=}} \EQ\Big[ \exp\Big( t\sum_{m\in T_{(n)},m_1=1} \zeta^{(n-1)}_{\cT(m)} \Big) \prod_{m'\in T_{(n-1)}} \IND_{E^\zeta_{\cT_1(m'),n-1}}\\
        &\hspace{4em} \times \exp\Big( t\sum_{m\in T_{(n)},m_1=2} \zeta^{(n-1)}_{\cT(m)} \Big)  \prod_{m'\in T_{(n-1)}} \IND_{E^\zeta_{\cT_2(m'),n-1}} \Big]+ e^{t2^n} \cdot 2^n \varepsilon_{n-1} \\
        &\underset{\eqref{eq:tree-distance}}{\overset{\eqref{eq:local}}{\leq}} \EQ\Big[ \exp\Big( t\sum_{m\in T_{(n)},m_1=1} \zeta^{(n-1)}_{\cT(m)} \Big)\Big] \cdot \EQ\Big[ \exp\Big( t\sum_{m\in T_{(n)},m_1=2} \zeta^{(n-1)}_{\cT(m)} \Big)\Big]\\
        &\hspace{4em} + e^{t2^n} \cdot 2^n \varepsilon_{n-1}, 
        \end{split}
    \end{equation}
    where in the last step, we first bound from above the indicator functions by one, and then use the independence between $\{\zeta^{(n-1)}_{\cT(m)}: m\in T_{(n)},m_1=1\}$ and $\{\zeta^{(n-1)}_{\cT(m)}: m\in T_{(n)},m_1=2\}$ by the locality assumption \eqref{eq:local} and the distance between boxes (see Lemma~\ref{lem:tree-geometry}). Indeed, for any $m\in T_{(n)}$ with $m_1=1$ and $m'\in T_{(n)}$ with $m'_1=2$, by~\eqref{eq:local}, $\zeta_{\cT(m)}^{(n-1)}$ and $\zeta_{\cT(m')}^{(n-1)}$ are independent if $B(\cT(m),3L_{n-1}) \cap B(\cT(m'), 3L_{n-1}) = \varnothing$, while we actually have $|\cT(m)-\cT(m')|_\infty \geq 90L_{n-1}$ by \eqref{eq:tree-distance}. 
    We let  
    \begin{equation}\label{eq:b_n}
        b_n(t) = e^{-t2^n} \cdot \sup_{x\in\bbL_n} \sup_{\cT\in \Lambda_{n,x}} \EQ \Big[ \exp\Big( t\sum_{m\in T_{(n)}} \zeta^{(n)}_{\cT(m)} \Big)\Big].   
    \end{equation}
    By taking the supremum and multiplying both sides by $e^{-t2^n}$ in~\eqref{eq:recursion-1}, we obtain 
    \begin{equation}\label{eq:b_n-recursion}
        b_n(t) \leq b_{n-1}(t)^2 + 2^n \varepsilon_{n-1}. 
    \end{equation}
    We write 
    \begin{equation}\label{eq:u_n}
        u_n(t) = b_n(t)^{2^{-n}}. 
    \end{equation}
    The recursion \eqref{eq:b_n-recursion} implies, using $(a+b)^s \leq a^s + b^s$ for $0<s<1$, that 
    \begin{align}\label{eq:u_n-recursion}
        u_n(t) \leq (b_{n-1}(t)^2 + 2^n \varepsilon_{n-1})^{2^{-n}} \leq u_{n-1}(t) + (2^n \varepsilon_{n-1})^{2^{-n}},
    \end{align}
    which gives (recall that $\varepsilon_j$ is a shorthand for $\varepsilon(L_j)$ defined in~\eqref{eq:E_n})
    \begin{equation}\label{eq:u_n-sum}
        \begin{aligned}
        u_n(t) &\leq u_0(t) + \sum_{j=0}^\infty (2^{j+1} \varepsilon_{j})^{2^{-(j+1)}} 
        \underset{\eqref{eq:L_n}}{\overset{\eqref{eq:E_n}}{\leq}} u_0(t) + c \sum_{j=0}^\infty \exp\Big(-c' L_0^\beta \big(\tfrac{\ell_0^\beta}{2}\big)^j \Big) \\
        &\stackrel{\ell_0^\beta \geq 3}{\leq} u_0(t)+ c\exp(-c'L_0^\beta) \stackrel{\mathrm{def.}}{=}  u_0(t) + \Delta(L_0).  
    \end{aligned}
    \end{equation}
    Note that 
    \begin{equation}
       u_0(t)\overset{\eqref{eq:u_n}}{=}b_0(t) \underset{\eqref{eq:identical-distributed}}{\overset{\eqref{eq:b_n}}{=}} e^{-t}\bbQ_p[\zeta^{(0)}_0=0]+ \bbQ_p[\zeta^{(0)}_0=1]. 
    \end{equation}
    The result follows by combining \eqref{eq:u_n-sum} and the definition of $b_n(t)$ and $u_n(t)$.
\end{proof}

In the next step, we study a field of possibly non-local Bernoulli random variables that can be approximated by the multiscale fields that we considered in Assumption \ref{ass:local-rv}. 

\begin{definition}\label{def:approximable}
    Let $L_0 \geq 1$, $\ell_0\geq 100$, and $(\zeta_x)_{x\in \bbL_0}$ a collection of Bernoulli random variables on the probability space $(\Omega_{\mathrm{bond}},\cF_{\mathrm{bond}}, \bbQ_p)$ indexed by $\bbL_0$. 
    We say $(\zeta_x)_{x\in \bbL_0}$, is $(L_0,\ell_0)$-\emph{approximable} (with decay exponent $\beta$) if 
    \begin{equation}\label{eq:def-appro-1}
        \begin{minipage}{0.9\textwidth}
            there exist collections $(\zeta_x^{(n)})_{x\in\bbL_0, n\geq 0}$ satisfying Assumption \ref{ass:local-rv} with parameters $L_0,\ell_0$,
        \end{minipage}
    \end{equation}
    \begin{equation}\label{eq:def-appro-2}
        \begin{minipage}{0.9\textwidth}
            for $n \geq 0$, $x\in \bbL_0$, there exists an event $F^\zeta_{n,x}$ such that $\zeta_x \IND_{F^\zeta_{n,x}} = \zeta^{(n)}_x  \IND_{F^\zeta_{n,x}}$, 
        \end{minipage}
    \end{equation}
    \begin{equation}\label{eq:def-appro-3}
        \begin{minipage}{0.9\textwidth}
            and the events $(F_{n,x}^\zeta)_{x\in \bbL_0}$ fulfill
    \begin{equation*}
        \bbQ_p[(F_{n,x}^\zeta)^c]\leq \varepsilon'(L_n) \stackrel{\mathrm{def.}}{=}  c_{15}\exp(-c_{16} L_{n}^{\beta}),\text{ where $\beta$ is the same as in~\eqref{eq:E_n}.}
    \end{equation*}  
        \end{minipage}
    \end{equation}
    We call the collections $(\zeta_x^{(n)})_{x\in\bbL_0, n\geq 0}$ \textit{approximating fields}.
\end{definition}

\begin{lemma}\label{lem:approximable}
    Suppose a collection of Bernoulli random variables $(\zeta_x)_{x\in \bbL_0}$ is $(L_0,\ell_0)$-approximable.  Then, for $n\geq 1$, $x\in \bbL_n$, and $\cT\in \Lambda_{n,x}$
    \begin{equation}\label{eq:approximable-bound}
        \EQ \Big[ \exp\big( t\sum_{m\in T_{(n)}} \zeta_{\cT(m)} \big)\Big] \leq  \EQ \Big[ \exp\big( t\sum_{m\in T_{(n)}} \zeta^{(n)}_{\cT(m)} \big)\Big] + e^{t2^n} \cdot 2^n \varepsilon'(L_n). 
    \end{equation}  
\end{lemma}

\begin{proof}
    For any $n\geq 1$ and any proper tree embedding $\cT$ at some root $x \in \bbL_n$ with depth $n \geq 0$, we let  
    \begin{equation}
        F_{n,\cT} = \bigcap_{m \in T_{(n)}} F^\zeta_{\cT(m), n}. 
    \end{equation}
    By \eqref{eq:def-appro-3} we see that $\bbQ_p[(F_{n,\cT})^c] \leq 2^n \varepsilon'(L_n)$, and by~\eqref{eq:def-appro-2}, we have under $F_{n,\cT}$ that 
    \begin{equation}
        \sum_{m\in T_{(n)}} \zeta_{\cT(m)} = \sum_{m\in T_{(n)}} \zeta_{\cT(m)}^{(n)}. 
    \end{equation}
    The result follows by decomposing the expectation on the left hand side of \eqref{eq:approximable-bound} into the events $F_{n,\cT}$ and $F_{n,\cT}^c$ and bounding from above the indicator functions by one. 
\end{proof}

We now state our main result of the present Section. It states that for $\bbQ_p$-a.e.~$\omega \in \Omega_{\mathrm{bond}}$, the proportion of boxes with $\zeta_{\cT(m)}=1$ among the leaves $m\in T_{(n)}$ is at most $\rho \in (0,1)$ for all proper tree embeddings at root $x$ (which can be chosen over a potentially growing subset of $\bbL_n$) when $n$ is large enough, assuming that the approximating Bernoulli field at the lowest level has a sufficiently small parameter. 

\begin{theorem}\label{thm:general-proportion}
   Let $\rho \in (0,1)$ and define $
        t = 3\rho^{-1}\log \cC_{d,\ell_0}$. Suppose a collection of Bernoulli random variables $(\zeta_x)_{x\in \bbL_0}$ is $(L_0,\ell_0)$-approximable with an approximating field $(\zeta_x^{(n)})_{x\in \bbL_0, n \geq 0}$ (see~\eqref{eq:def-appro-1}) for $\ell_0 \geq 3^{1/\beta} \vee 100$ and $L_0 \geq 1$ such that for 
      \begin{equation}\label{eq:theta-L0}
        \vartheta_{L_0}\stackrel{\mathrm{def}.}{=}\bbQ_p[\zeta_0^{(0)}=1],
    \end{equation} 
 one has 
 \begin{equation}
 \label{eq:Two-conditions-on-L0}
\begin{cases}
& \log(1+ e^t(\vartheta_{L_0} + \Delta(L_0)) \leq \log \mathcal{C}_{d,\ell_0}, \text{ and } \\
&  t-t\rho-cL_0^\beta \leq -3\log \cC_{d,\ell_0}.
 \end{cases}
\end{equation}   
    Let $(D_n)_{n\geq 1}$ be any deterministic sequence of finite sets such that $D_n\subseteq\bbL_n$ and
    \begin{equation}\label{eq:window-growth}
       \lim_{n\to\infty} 2^{-n}\log |D_n| = 0.
    \end{equation}
    Then, 
    \begin{equation}\label{eq:result-general-proportion}
        \begin{minipage}{0.9\textwidth}
            for $\bbQ_p$-a.e.~$\omega$, $\displaystyle\sup_{x\in D_n}\sup_{\cT\in \Lambda_{n,x}} \frac{1}{2^n} \sum_{m\in T_{(n)}} \zeta_{\cT(m)}(\omega) < \rho$ for all but finitely many $n$. 
        \end{minipage}
    \end{equation}
\end{theorem}

\begin{proof}

    Let $n\geq 1$, $x\in D_n$, and $\cT\in \Lambda_{n,x}$. 
    By the exponential Chebyshev inequality and Lemmas \ref{lem:multi-itera} and \ref{lem:approximable}, we obtain for $t = 3\rho^{-1}\log \cC_{d,\ell_0}$, 
    \begin{equation}\label{eq:MC-esti}
        \begin{split}
            &\bbQ_p \Big[ \sum_{m\in T_{(n)}} \zeta_{\cT(m)} \geq \rho 2^n \Big] \leq \frac{ \EQ\Big[ \exp\Big( t\sum_{m\in T_{(n)}} \zeta_{\cT(m)} \Big) \Big]}{\exp(t\rho 2^n)}  \\
        &\leq \exp\Big \{ \left( (1-\rho)t+\log((1-\vartheta_{L_0})e^{-t}+\vartheta_{L_0}+\Delta(L_0)) \right)2^n \Big\} + e^{t(1-\rho)2^n} 2^n \varepsilon'(L_n). 
        \end{split}
    \end{equation}
   We now note that
    \begin{equation}
    \label{eq:aux-calc}
    \begin{split}
        (1-\rho)t&+\log((1-\vartheta_{L_0})e^{-t}+\vartheta_{L_0}+\Delta(L_0)) \leq (1-\rho)t + \log(e^{-t}+\vartheta_{L_0}+\Delta(L_0))\\
        & = -\rho t + \log(1+e^t (\vartheta_{L_0}+\Delta(L_0))) = -3\log \cC_{d,\ell_0} + \log(1+e^t (\vartheta_{L_0}+\Delta(L_0))) \\
        & \stackrel{\eqref{eq:Two-conditions-on-L0}}{\leq} -2 \log \mathcal{C}_{d,\ell_0}.
        \end{split}
    \end{equation}
   
   Upon inserting~\eqref{eq:aux-calc} into the first term on the right hand side of~\eqref{eq:MC-esti}, we obtain 
    \begin{equation}\label{eq:1st-term}
        \exp\Big \{  \left( (1-\rho)t+\log((1-\vartheta)e^{-t}+\vartheta+\Delta(L_0)) \right) 2^n \Big\} \leq \exp(-2\log\cC_{d,\ell_0} \cdot 2^n).
    \end{equation}
     
    We turn to the second term on the right hand side of~\eqref{eq:MC-esti}. For the latter we have
    \begin{equation}
    \begin{split}
        e^{t(1-\rho)2^n} 2^n \varepsilon'(L_n)& \underset{\ell_0^\beta \geq 3}{\overset{\eqref{eq:def-appro-3},\eqref{eq:L_n}}{\leq}} c \exp\Big( \log 2\cdot n + \left(t-t\rho-cL_0^\beta\right) 2^n \Big) \\
        &  \stackrel{\eqref{eq:Two-conditions-on-L0}}{\leq} c \exp\Big( \log 2\cdot n -3 \log \mathcal{C}_{d,\ell_0} \cdot  2^n \Big).
            \end{split}
    \end{equation}  
    Since $\log \cC_{d,\ell_0}\cdot 2^n > \log 2\cdot n$, combining with \eqref{eq:1st-term}, we obtain by following~\eqref{eq:MC-esti} that 
    \begin{equation}\label{eq:deviation-estimate}
        \bbQ_p \Big[ \sum_{m\in T_{(n)}} \zeta_{\cT(m)} \geq \rho 2^n \Big] \leq  c\exp(-2\log\cC_{d,\ell_0} \cdot 2^n). 
    \end{equation}

    Finally, union bounds over the tree embeddings and the roots in $D_n$ give
    \begin{equation}
    \begin{split}
            &\sum_{n=1}^\infty \bbQ_p\Big[\sup_{x\in D_n}\sup_{\cT\in \Lambda_{n,x}} \frac{1}{2^n} \sum_{m\in T_{(n)}} \zeta_{\cT(m)} \geq \rho\Big] 
        \underset{\eqref{eq:deviation-estimate}}{\overset{\eqref{eq:Lambda_n-x-number}}{\leq}} \sum_{n=1}^\infty c|D_n|\cC_{d,\ell_0}^{2^n} \exp(-2\log \cC_{d,\ell_0} \cdot 2^n) \\
        &\qquad= \sum_{n=1}^\infty c|D_n|\exp(- \log \cC_{d,\ell_0} \cdot 2^n) \overset{\eqref{eq:window-growth}}{<} \infty.
     \end{split}
    \end{equation}
    This gives \eqref{eq:result-general-proportion} by an application of the Borel-Cantelli Lemma. 
\end{proof}

\section{Existence of a subcritical phase}
\label{sec:Subcritical-Phase}

In the present Section, we establish the existence of subcritical regimes for both the level sets of the Gaussian free field, and the vacant set of random interlacement defined on a typical realization of the Bernoulli percolation cluster $\mathcal{C}_\infty$. \medskip

We first apply the technical tools developed in Section~\ref{sec:Quenched-Unif-Estimates} to ``decouple'' the local irregularities of the percolation cluster $\cC_\infty$ encoded by the random variables $(S_x)_{x\in \bbZ^d}$ in \eqref{eq:S_x}. Roughly speaking, we show that for $\bbQ_p$-a.e.~$\omega$, along any tree embedding into a box of size $L_n = L_0 \ell_0^n$, $n \geq 1$ large enough, the sizes of the regions where certain irregularities might occur are uniformly bounded for at least half of the boxes corresponding to the leaves. We are then able to show that the capacity of any subset of points in these boxes can be bounded from below by the same order as the number of points, which leads to the results. 
\medskip

\begin{lemma}\label{lem:verify-Z-x}
    Let $L_0 \geq 1$ and $\ell_0\geq 100$. For $x\in \bbL_0$, we define 
    \begin{equation}\label{eq:Z_x}
        Z_x = \IND \Big\{\sup_{z\in B(x,L_0)} S_z > M(L_0)\Big\},\quad  \text{ with $M(L_0)= (\log L_0)^{2/\eta_{\mathrm{hk}}}$}.
    \end{equation}
    
    Then $(Z_x)_{x\in \bbL_0}$ is $(L_0,\ell_0)$-approximable with decay exponent $\eta_{\mathrm{hk}}$. 
\end{lemma}

\begin{proof}
    We define for $n \geq 0$ and $x\in \bbL_0$ the random variables
    \begin{equation}\label{eq:SxL}
        \begin{aligned}
        &S_x^{(L_n)} = \sup\{L \leq L_n:\text{$\cL(Q)$ fails for some cube $Q$ containing $x$ and with $s(Q)=L$}\} \\
        &\text{(with the convention $\sup \varnothing = 0$), and} \\
        &Z_x^{(n)} = \IND \Big\{\sup_{z\in B(x,L_0)} S_z^{(L_n)} > M(L_0) \Big\}.
    \end{aligned}
    \end{equation}
    We first show that the collections $(Z_x^{(n)})_{x\in \bbL_0}$, $n \geq 0$, satisfy Assumption \ref{ass:local-rv}. 
    It follows by the translation invariance of $\bbQ_p$ that the random variables $(Z_x^{(n)})_{x\in \bbL_0}$ are identically distributed. Also, by~\eqref{eq:LQ-measurable} in Lemma~\ref{lem:Barlow-HK}, we see that $Z_x^{(n)}$ is $\cF_{E(B(x,3L_n))}$-measurable. 
    It remains to verify~\eqref{eq:stable} and~\eqref{eq:E_n} in Assumption~\ref{ass:local-rv}. We define for $n \geq 0$ and $x\in \bbL_0$, 
    \begin{equation}
        E^Z_{n,x} = \bigcap_{z\in B(x,L_0)} \{S_z \leq L_n \},
    \end{equation}
    which will play the role of the events in~\eqref{eq:stable}. We claim that 
    \begin{equation}\label{eq:S_x-stable}
        S_x^{(L_{n+1})} \IND_{\{S_x \leq L_n\}} = S_x^{(L_{n})} \IND_{\{S_x\leq L_n\}} = S_x \IND_{\{S_x\leq L_n\}}.
    \end{equation}
    Indeed, under $\{S_x\leq L_n\}$, $\cL(Q)$ holds for every cube $Q$ containing $x$  with $s(Q)> L_n$, and thus the largest number $L$ such that $\cL(Q)$ fails (with $s(Q)=L$ contain $x$) remains unchanged when considering the largest number with that property among any of the sets $\{L\leq L_{n+1}\}$, $\{L \leq L_n\}$, or $\{L \geq 1\}$. 
    It follows from~\eqref{eq:S_x-stable} that 
    \begin{equation}\label{eq:Z-approximable}
        Z_{x}^{(n+1)} \IND_{E^Z_{n,x}} = Z_{x}^{(n)} \IND_{E^Z_{n,x}} = Z_x \IND_{E^Z_{n,x}}, 
    \end{equation}
    which gives~\eqref{eq:stable}. 
    To see~\eqref{eq:E_n}, note that  
    \begin{equation}\label{eq:E-Z-decay}
        \bbQ_p\big[(E^Z_{n,x})^c\big] \leq cL_0^d \cdot \bbQ_p[S_z>L_n] \overset{\eqref{eq:S_x-decay}}{\leq} c\exp(-c' L_n^{\eta_{\mathrm{hk}}}). 
    \end{equation}
    Therefore $(Z_x^{(n)})_{x\in \bbL_0}$ satisfies Assumption~\ref{ass:local-rv}. Finally, it follows by the second equality in~\eqref{eq:Z-approximable} and~\eqref{eq:E-Z-decay} that $(Z_x)_{x\in \bbL_0}$ is $(L_0,\ell_0)$-approximable (taking $F_{n,x}^Z=E_{n,x}^Z$ in Definition~\ref{def:approximable}). 
\end{proof}

Note that again by \eqref{eq:S_x-decay}
\begin{equation}
\label{eq:Observation-smallest-scale}
    \bbQ_p[Z_0^{(0)} =1 ] \leq cL_0^d\bbQ_p[S_x \geq M(L_0)] \to 0 \quad \text{as $L_0\to \infty$}.
\end{equation}

\begin{corollary}\label{coro:subcritical}
    Let $\ell_0 \geq 3^{1/\eta_{\mathrm{hk}}} \vee 100$ be fixed. There exists $L_* = L_*(\ell_0)$ such that for $\bbQ_p$-a.e.~$\omega$, and $L_0 \geq L_*(\ell_0)$ for $n\geq n_0(\omega,\ell_0,L_0)$ large enough and any choice of $\cT\in \Lambda_{n,0}$, 
    \begin{equation}
    \label{eq:Result-Cor-4.2}
        \sum_{m\in T_{(n)}} Z_{\cT(m)} \leq 2^{n-1}. 
    \end{equation}
\end{corollary}
\begin{proof}
Let $\rho=\frac{1}{2}$ and $D_n=\{0\}$. By~\eqref{eq:Observation-smallest-scale}, and since $\Delta(L_0) \to 0$ as $L_0 \to \infty$ by~\eqref{eq:Delta-L0}, we can choose any fixed $L_0$ larger than $L_*(\ell_0) \stackrel{\mathrm{def.}}{=} \inf\{L_0 \geq 1 \, : \, \text{\eqref{eq:Two-conditions-on-L0} is fulfilled} \} $. An application of Theorem~\ref{thm:general-proportion} then gives~\eqref{eq:Result-Cor-4.2}.
\end{proof}

We now apply the corollary to provide a capacity estimate which is essential in the proof of the existence of subcritical regimes for both models. 

\begin{lemma}\label{lem:capacity-lower-bound}
    Let $\ell_0 \geq 3^{1/\eta_{\mathrm{hk}}} \vee 100$, 
    \begin{equation}\label{eq:L_0-range}
    L_0\geq L_*(\ell_0) \vee \min \{L \geq 1: (c_7\vee c_8) M(\tilde{L}) \leq \tilde{L},\text{ for all } \tilde{L}\geq L\},
    \end{equation}
    and $n \geq n_0(\omega,\ell_0,L_0)$ as in Corollary~\ref{coro:subcritical}. Let $\cT\in \Lambda_{n,0}$.  We define 
    \begin{equation}\label{eq:non-emptyness}
        \widehat{T}_{(n),\cT}=\{m\in T_{(n)}:Z_{\cT(m)}=0\} \quad \text{(with $|\widehat{T}_{(n),\cT}| \geq 2^{n-1}$ by Corollary~\ref{coro:subcritical})}.
    \end{equation}
    \begin{equation}
        \begin{minipage}{0.9\textwidth}
            For every $m\in\widehat T_{(n),\cT}$, we assume $B(\cT(m),L_0)\cap\cC_\infty\neq\varnothing$,
        \end{minipage}
    \end{equation}
    \begin{equation}\label{eq:X_T}
        \begin{minipage}{0.9\textwidth}
           choose $y_m\in B(\cT(m),L_0)\cap\cC_\infty$ for any $m \in \widehat{T}_{(n),\mathcal{T}}$ and define
            \[\cX_\cT=\{y_m:m\in\widehat T_{(n),\cT}\}.\]
        \end{minipage}
    \end{equation}
    
    Then
    \begin{equation}\label{eq:capacity-X_T}
         \capa^\omega(\cX_\cT) \geq c_{15}L_0^{-1} 2^n. 
    \end{equation}
\end{lemma}

\begin{proof}
    We first recall the definition of $(Z_{x})_{x\in \bbL_0}$ in~\eqref{eq:Z_x} and note that $Z_x=0$ implies $S_z \leq M(L_0)$ for all $z\in B(x,L_0)$. By \eqref{eq:on-diag-Green} in Lemma~\ref{lem:Green-function-estimate} we have for $m\in \widehat{T}_{(n),\cT}$ the bound
    \begin{equation}\label{eq:on-Green-by-M}
        g^\omega(z,z) \leq c_7 M(L_0), \quad \text{for all $z\in B(\cT(m),L_0) \cap\cC_\infty$},
    \end{equation}
    and by~\eqref{eq:off-diag-Green},
    \begin{equation}\label{eq:off-Green-by-L}
         g^\omega(z,y) \leq \frac{c_{10}}{|z-y|_\infty^{d-2}}, \text{ for $z\in B(\cT(m),L_0)\cap \cC_\infty,y\in \cC_\infty$ with $|y-z|_\infty \geq c_8 M(L_0)$}. 
    \end{equation}
    
    Note that by \eqref{eq:L_0-range}, $(c_7\vee c_8) M(L_0)\leq L_0$, and Lemma~\ref{lem:tree-geometry} yields that for $m \neq m' \in \widehat{T}_{(n),\cT}$, $|\cT(m)-\cT(m')|_\infty \geq 90L_0\,(\geq c_8M(L_0))$. For any $y_m\in \cX_\cT$,
    \begin{align}
        \sum_{y\in \cX_\cT} g^\omega(y_m,y) \leq \sum_{k=0}^n \sum_{m' \in T_{(n)}^{m,k}} g^\omega(y_m,y_{m'}) 
        \underset{\eqref{eq:off-Green-by-L}}{\overset{\eqref{eq:on-Green-by-M}}{\leq}} L_0+ \sum_{k=1}^n c_{10} L_{k-1}^{2-d} |T_{(n)}^{m,k}| 
        \overset{\eqref{eq:sub-division-tree}}{\leq} c L_0. 
    \end{align}
    The capacity estimate \eqref{eq:capacity-X_T} follows by \eqref{eq:capacity-bounds}. 
\end{proof}

\begin{proof}[\textbf{Proof of Theorem~\ref{thm:PH-GFF}}]
    We first show that $\alpha^\omega_*$ equals a deterministic value for $\mathbb{Q}_p$-a.e.~$\omega$. As in the proof of~\cite[Proposition 3.1]{CN2021disconnection}, we note that the event $\{\text{$E_\omega^{\geq \alpha}$ contains an infinite cluster} \}$ is translation-invariant with respect to $(t_x)_{x \in \bbZ^d}$  (recall~\eqref{eq:Def-t-x} for the definition). Therefore, one has
    \begin{equation}
    \bbP^G_\omega[\text{$E_\omega^{\geq \alpha}$ contains an infinite cluster}] = \bbP^G_{\tau_x\omega}[\text{$E_{\tau_x\omega}^{\geq \alpha}$ contains an infinite cluster}],
    \end{equation}
    for $\omega \in \Omega_0$ (see~\eqref{eq:shift-t}) and all $x \in \mathbb{Z}^d$. Since $\mathbb{Q}_p$ is stationary and ergodic, we obtain that $\alpha_\ast^\omega$ is $\mathbb{Q}_p$-a.s.~constant.

    As mentioned before, we have $\alpha^\omega_* \geq 0$ by \cite[Proposition A.2]{abacherli2018} for all $\omega\in \Omega_0$, which implies $\alpha_*\geq 0$. \medskip

    We now prove $\alpha_*<\infty$. For $\alpha \in \bbR$ and $n\geq 1$, by the construction of the proper tree embedding at root $0\in \bbL_n$ (see Lemma~\ref{lem:renormalization}) similarly as in~\cite{rath2015transition},
    \begin{equation}\label{eq:phas-tran-1}
        \begin{split}
            \bbP^G_\omega &[S(0,L_n-1)\overset{\geq \alpha}\longleftrightarrow \infty] \leq 
         \bbP^G_\omega\Big[\bigcup_{\cT\in \Lambda_{n,0}} \bigcap_{m\in T_{(n)}} \bigcup_{z\in B(\cT(m),L_0)\cap \cC_\infty}\{\varphi_z \geq \alpha\}\Big]  \\
        &\overset{\eqref{eq:Lambda_n-x-number}}{\leq }\cC_{d,\ell_0}^{2^n} \max_{\cT\in \Lambda_{n,0}}  \bbP^G_\omega \Big[\bigcap_{m\in T_{(n)}} \bigcup_{z\in B(\cT(m),L_0)\cap \cC_\infty}\{z\in E_\omega^{\geq \alpha}\}\Big].  
        \end{split}
    \end{equation}
    We can assume without loss of generality that the event inside the above probability is non-empty, and then \eqref{eq:non-emptyness} is automatically satisfied. 
    Let the conditions on $\ell_0, L_0, n$ in Lemma~\ref{lem:capacity-lower-bound} hold. We can bound the right hand side~of~\eqref{eq:phas-tran-1} from above by 
    \begin{equation*}
        \cC_{d,\ell_0}^{2^n} \max_{\cT\in \Lambda_{n,0}}  \bbP^G_\omega \Big[\bigcap_{m\in \widehat{T}_{(n),\cT}} \bigcup_{z\in B(\cT(m),L_0)\cap \cC_\infty}\{z\in E_\omega^{\geq \alpha}\}\Big].
    \end{equation*}
    Note that the number of choices of $\cX_\cT$ is at most $(2L_0+1)^{d2^n}$, so that 
    \begin{equation}\label{eq:phase-trans-2}
       \begin{split}
        \bbP_\omega^G &[S(0,L_n-1)\overset{\geq \alpha}\longleftrightarrow \infty] 
        \leq \cC_{d,\ell_0}^{2^n} (2L_0+1)^{d2^n} \max_{\cT\in \Lambda_{n,0}} \bbP_\omega^G [\cX_\cT \subseteq E^{\geq \alpha}]\\
        &\underset{\eqref{eq:capacity-X_T}}{\overset{\eqref{eq:GFF-capa}}{\leq }}\cC_{d,\ell_0}^{2^n} (2L_0+1)^{d2^n} \exp\Big(-\tfrac{c_{15}}{2}  L_0^{-1}\alpha^2 2^n\Big) \\
        &= \exp\Big\{ (\log \cC_{d,\ell_0} + d\log(2L_0+1) - \tfrac{c_{15}}{2}  L_0^{-1}\alpha^2) 2^n \Big\}. 
    \end{split} 
    \end{equation} 
    
    If $\alpha^2> 2c_{15}^{-1} L_0 (\log \cC_{d,\ell_0} + d\log(2L_0+1))$, then the coefficient in front of $2^n$ is negative. We then obtain 
    $\lim_{n\to\infty} \bbP_\omega^G [S(0,L_n-1)\overset{\geq \alpha}\longleftrightarrow \infty]=0$. Finally, 
    \begin{equation}\label{eq:phase-trans-3}
        \begin{split}
        \bbP_\omega^G &[\{E^{\geq \alpha}\text{ contains an infinite cluster}\}] \leq \bbP_\omega^G \left[\bigcup_{n=1}^\infty \bigcap_{i=n}^\infty \{S(0,L_i-1)\overset{\geq \alpha}\longleftrightarrow \infty\}\right] \\
        &\leq \lim_{n\to\infty} \bbP_\omega^G [S(0,L_n-1)\overset{\geq \alpha}\longleftrightarrow \infty]=0.
    \end{split}
    \end{equation}
    This shows $\alpha_* <\infty$. 
\end{proof}

The same argument also gives the non-randomness and the finiteness of $u_*^\omega$.

\begin{proof}[\textbf{Proof of Theorem~\ref{thm:PH-RI}, Part I}]
    The fact that $u_*^\omega$ equals a deterministic value $u_*$ follows similarly as the argument in the proof of Theorem~\ref{thm:PH-GFF}. We also have, as in~\eqref{eq:phas-tran-1}, 
\begin{equation}\label{eq:phase-trans-1'}
    \begin{split}
        \bbP^I_\omega &[S(0,L_n)\overset{\cV_\omega^u}\longleftrightarrow \infty] \leq \bbP^I_\omega\Big[\bigcup_{\cT\in \Lambda_{n,0}} \bigcap_{m\in T_{(n)}} \bigcup_{z\in B(\cT(m),L_0)\cap \cC_\infty}\{z \in \cV_\omega^u\}\Big] \\
    &\leq \cC_{d,\ell_0}^{2^n} (2L_0+1)^{d2^n} \max_{\cT\in \Lambda_{n,0}}  \bbP_\omega^I [\cX_\cT \subseteq \cV_\omega^u].
    \end{split}
\end{equation}
We can then apply Lemma~\ref{lem:capacity-lower-bound} in the same way, using \eqref{eq:law-Vu} in place of \eqref{eq:GFF-capa} to \eqref{eq:phase-trans-1'}, to conclude by~\eqref{eq:phase-trans-3} that 
\begin{equation}\label{eq:u-star-finite}
        \text{$\bbQ_p$-a.s., $u_*^\omega$ equals a deterministic value $u_*<\infty$.}
\end{equation}
We will prove in the next section that $u_*$ is strictly positive. 
\end{proof}

\section{Existence of a supercritical phase}

\label{sec:Supercritical-Phase}

In this Section, we establish the existence of a supercritical phase for the vacant set of random interlacements $\cV_\omega^u$ on a typical realization of the Bernoulli percolation cluster $\mathcal{C}_\infty$.  \medskip

We briefly comment on the strategy of the proof, which involves two stages of renormalization. In the first stage, we ``coarsen'' the infinite cluster $\cC_\infty$ into a site percolation cluster on a renormalized lattice $\bfZ^d = k\bbZ^d$  (defined in~\eqref{eq:Renormalized-lattice}) which dominates a high-density supercritical Bernoulli site percolation model (``Renormalization I''). This one-step renormalization, which heavily relies on a similar construction in~\cite{barlow2004RWpercolation}, guarantees the existence of an infinite cluster $\cC_\infty^\bfF$ on the ``coarse-grained'' plane $\bfF$ (defined in~\eqref{eq:Renormalized-plane}), see Lemma~\ref{lem:percolation-plane}. 
The random interlacements $\mathcal{I}^u_\omega$ on $\cC_\infty$ that intersect cubes of size $k$ anchored at points in $\cC_\infty^\bfF$ give rise to a subset of $\cC_\infty^\bfF$ which we call $u$-occupied vertices (with the complement in  $\cC_\infty^\bfF$ consisting of vertices we call $u$-vacant), see Definition~\ref{def:good-and-vacant}. In the second stage, we prove that $u$-vacant vertices percolate in the plane, utilizing the renormalization scheme described in Section~\ref{sec:Quenched-Unif-Estimates} and Theorem~\ref{thm:general-proportion}. In addition to the estimates related to random walks, the local geometry is also essential here for ensuring a large enough contribution of the cost of creating $u$-occupied vertices. In the end, we translate the result back to the original lattice and prove that $\cV^u_\omega$ percolates in a slab for small enough $u>0$, for $\mathbb{Q}_p$-a.e.~$\omega$.

\subsection{Renormalization I: Coarse-graining of $\mathcal{C}_\infty$}

We introduce a coarse-graining scheme taken from \cite{barlow2004RWpercolation}. Let $k\geq 17$ and consider the renormalized lattice 
\begin{equation}
\label{eq:Renormalized-lattice}
    \mathbf{Z}^d = k\mathbb{Z}^d. 
\end{equation}
We denote by $|\,\cdot\,|_{1,\bfZ^d}$ and $|\,\cdot\,|_{\infty,\bfZ^d}$ the $\ell_1$- and $\ell_\infty$-norms on $\bfZ^d$. 
There exists a tiling of $\bbZ^d$ by disjoint cubes indexed by vertices in the renormalized lattice:
\begin{equation}
    \bbZ^d=\bigcup_{x\in \bfZ^d} Q_k(x), \quad \text{where }Q_k(x)=\{y\in \bbZ^d: x_i\leq y_i<x_i+k, \forall 1\leq i\leq d\},\, x \in \bfZ^d.
\end{equation}

We introduce some further notation concerning Bernoulli bond percolation. For a subgraph $(A,E_A)$ of $(\bbZ^d, \bbE_d)$, we define the graph distance $d_{E_A}(x,y)$ for $x,y\in \bbZ^d$ to be the smallest number $j$ such that there exists a path $\gamma : \{0,...,j\} \rightarrow A$ of length $j$ connecting $x$ and $y$ with $\{\gamma(k),\gamma(k+1)\} \in E_A$ for all $0 \leq k <j$, and if there is no such path we define the value to be $\infty$. Recall that for $\omega \in \Omega_{\mathrm{bond}}$, $\cC$ denotes the set of all open edges. For a cube $Q$, we define the \emph{open $Q$-cluster containing $x$} by $\cC_Q(x)= \{y\in Q: d_{E(Q)\cap \cC}(x,y)<\infty\}$, namely the set of points connected to $x$ by an open path within $Q$. We write $\cC^\vee(Q)$ for the largest open $Q$-cluster containing some point in $Q$ (with ties broken arbitrarily). 
\medskip

We then introduce a 2-dependent site percolation process considered in~\cite{barlow2004RWpercolation} (see also~\cite{antal1996chemical}).  

\begin{definition}\label{def:crossing}
    Consider a cube $Q$ of side-length $L \geq 16$.  
    \begin{equation}
        \begin{minipage}{0.9\textwidth}
            We say that an open $Q$-cluster $\cO$ in $Q$ is \emph{crossing} for a cube $Q'\subseteq Q$, if for all $d$ dimensions there exists an open path in $\cO\cap Q'$ that connects the two opposing faces of $Q'$.
        \end{minipage}
    \end{equation}
    \begin{equation}\label{eq:RQ}
        \begin{minipage}{0.9\textwidth}
            We let $\cR(Q)$ be the $\cF_{E(Q^+)}$-measurable event denoting the joint occurrence of the following: 
            \[
            \left\{
            \begin{aligned}
                \text{(i)}\, &\, \text{there exists a unique crossing cluster $\cO$ in $Q^+$ for $Q^+$},\\
                \text{(ii)} &\, \text{all open paths in $Q^+$ of diameter greater than $\frac{L}{8} $ are connected to $\cO$ in $Q^+$}, \\
                \text{(iii)} &\, \text{$\cO$ is crossing for each cube $Q'\subseteq Q$ with $s(Q')\geq \frac{L}{8}$},\\
                \text{(iv)} &\, \text{$\cC^\vee (Q)$ is crossing in $Q$ for $Q$ and $\cC^\vee (Q^+)$ is crossing in $Q^+$ for $Q^+$} 
            \end{aligned}
            \right.\]
            (note that in particular, $\cC^\vee(Q^+)$ is the unique crossing cluster for $Q^+$, and $\cC^\vee(Q)\subseteq \cC^\vee (Q^+)$).
        \end{minipage}
    \end{equation}
    \begin{equation}
        \begin{minipage}{0.9\textwidth}
            We define
            \[
            \phi(x)=\IND_{\cR(Q_k(x))} \quad \text{for $x\in \mathbf{Z}^d$},
            \] 
            which is a collection of Bernoulli random variables on $(\Omega_{\mathrm{bond}},\cF_{\mathrm{bond}}, \bbQ_p)$ indexed by $\bfZ^d$. 
            We say $x \in \bfZ^d$ is \emph{$\phi$-open} if $\phi(x)=1$ and \emph{$\phi$-closed} otherwise. 
        \end{minipage}
    \end{equation}
\end{definition} 

Note that, as stated in the proof of \cite[Lemma 2.8]{barlow2004RWpercolation}, 
\begin{equation}
    \text{$\phi(x)$ and $\phi(y)$ are independent if $|x-y|_{\infty,\bfZ^d} \geq 3$.} 
\end{equation}

\begin{lemma}[\cite{barlow2004RWpercolation}, Lemma~2.8]\label{lem:domination} 
    \begin{equation}\label{eq:RQ-decay}
        \begin{minipage}{0.9\textwidth}
            For a cube $Q$ with side-length $k \geq 16$, 
            \begin{equation*}
                \bbQ_p[\cR(Q)^c] \leq c_{17}\exp(-c_{18} k). 
            \end{equation*}
        \end{minipage}
    \end{equation}
    \begin{equation}\label{eq:domination}
        \begin{minipage}{0.9\textwidth}
            The process $(\phi(x))_{x\in \bfZ^d}$ stochastically dominates a Bernoulli site percolation on $\bfZ^d$ with parameter $q^*(k)$ for $k\geq 17$, and $q^*(k)\to 1$ as $k\to \infty$. 
        \end{minipage}
    \end{equation}
\end{lemma}

We define
\begin{equation}
\label{eq:Renormalized-plane}
    \bbF = \bbZ^2 \times \{0\}^{d-2} \quad \text{and} \quad \bfF=k\bbF = \bfZ^2\times \{0\}^{d-2}. 
\end{equation}
We then consider (by~\eqref{eq:domination}) 
\begin{equation}\label{eq:k0pd}
    k \geq k_0(p,d),
\end{equation}
where $k_0(p,d)$ is the smallest integer such that $q^*(k)>q_c(2)$ holds for all $k \geq k_0(p,d)$, with $q_c(2)$ denoting the critical parameter for the Bernoulli site percolation on $\bbZ^2$. 
By~\eqref{eq:domination}, we know that $(\phi(x))_{x\in \bfZ^d}$ dominates a Bernoulli site percolation which almost surely percolates in the plane $\bfF$. 

\begin{lemma}\label{lem:percolation-plane}
    There exists $\Omega_2\subseteq \Omega_0$ with $\bbQ_p[\Omega_2]=1$ such that for $\omega \in \Omega_2$, the site percolation process $(\phi(x))_{x\in \bfZ^d}$ contains a unique infinite $\phi$-open cluster, which we denote by $\cC_\infty^\bfF$.
\end{lemma}

\begin{proof}
    As the site percolation $(\phi(x))_{x\in \bfZ^d}$ dominates a supercritical Bernoulli site percolation with $q^*(k) > q_c(2)$, we are left to show the uniqueness. Since $\bbQ_p$ is a product measure, one can see from the definition of $\cR(Q)$ that $(\phi(x))_{x\in \bfZ^d}$ is stationary and has the \emph{finite energy} property (see, e.g., \cite[Definition~2]{burton-keane1989}). The uniqueness then follows by the Burton--Keane Theorem (see~\cite[Theorem~2]{burton-keane1989}). 
\end{proof}

In what follows, we write $x \overset{\phi}{\longleftrightarrow} y$ (resp.\ $\bfA \overset{\phi}{\longleftrightarrow} \bfB$) to indicate that $x$ and $y$ (resp.\ $\bfA$ and $\bfB$) are connected in $\bfF$ by a nearest-neighbor path (defined with $| \, \cdot \, |_{1,\bfZ^d}$) with all vertices $z$ being $\phi$-open.

\begin{proposition}\label{prop:renorm-I}
    Let $\omega \in \Omega_2$. The set $\bigcup_{x\in \cC_\infty^\bfF} \cC^\vee(Q_k(x))$ is a connected subgraph of $(\mathbb{Z}^d,\mathcal{C})$, thus forming an infinite open subgraph in $\mathbb{Z}^2\times \{0,\dots,k-1\}^{d-2}$.  
   Moreover, one has 
    \begin{equation}\label{eq:subset-C-infty}
        \bigcup_{x\in \cC_\infty^\bfF} \cC^\vee(Q_k(x)) \subseteq \cC_\infty. 
    \end{equation}
\end{proposition}

\begin{proof}
    The definition of $\cR(Q)$ in~\eqref{eq:RQ} gives 
    \begin{equation}\label{eq:C^V-connected}
        \begin{aligned}
            &\text{$x,y\in \bfF$ with $|x-y|_{1,\bfZ^d}=1$, and $\phi(x)=\phi(y)=1$} \\
            & \qquad \implies  \text{$\cC^\vee (Q_k(x))$ is connected to $\cC^\vee (Q_k(y))$ by an open path}. 
        \end{aligned}
    \end{equation}
    Therefore by the connectivity of $\cC_\infty^\bfF$, the set $\bigcup_{x\in \cC_\infty^\bfF} \cC^\vee(Q_k(x))$ forms an infinite connected subgraph in $\mathbb{Z}^2\times \{0,\dots,k-1\}^{d-2}$. Claim~\eqref{eq:subset-C-infty} follows by the uniqueness of the infinite open cluster of Bernoulli bond percolation when $\omega \in \Omega_0$.
\end{proof}

\subsection{Renormalization II: Tree embeddings in the coarse-grained lattice}
We now turn to the second stage of the renormalization procedure, which proceeds in a similar manner as in the proof of $u_*<\infty$ in Section~\ref{sec:Subcritical-Phase}, but on $\bfF$ instead of $\bbZ^d$. 

We begin by introducing some further notation. 
Let $x\in \bfF$ and $L\geq 1$ be an integer. We denote by $\bfB(x,L)=\{y\in \bfZ^d:|x-y|_{\infty,\bfZ^d} \leq L\}$ the closed box of radius $L$ on the space $\bfZ^d$ with respect to $| \, \cdot |_{\infty,\bfZ^d}$. 
We denote by 
\begin{equation}
    \bm{\frP}(x,L) = \{y\in \bfF:|x-y|_{\infty,\bfZ^d} \leq L\} \text{ and } \bm{\frF}(x,L) =\{y\in \bfF:|x-y|_{\infty,\bfZ^d} = L\} 
\end{equation}
the \emph{plate} and the \emph{frame} centered at $x$ with radius $L$, which are subsets of $\bfF$. 
The corresponding sets in the original lattice $\bbZ^d$ are denoted by
\begin{equation}
    \begin{split}
    &\frP^{x,L}_k = \bigcup_{z\in \bm{\frP}(x,L)} Q_k(z), \quad \frF^{x,L}_k =  \bigcup_{z\in \bm{\frF}(x,L)} Q_k(z) \\
    &\text{(note that $\frF^{x,L}_k$ is made up of four ``tubes'' of length $(2L+1)k$ and thickness $k$).}
    \end{split}
\end{equation}

Recall $(S_x)_{x\in \bbZ^d}$ in \eqref{eq:S_x}. We define $M(L)= (\log L)^{2/\eta_{\mathrm{hk}}}$ as in~\eqref{eq:Z_x}. 

\begin{definition} \label{def:good-and-vacant}\
    \begin{equation}\label{eq:plate-good}
        \begin{minipage}{0.9\textwidth}
           For $x\in \bfF,L\geq 1$, we say that a plate $\bm{\frP}(x,L) \subseteq \bfF$ is $\omega$-good if  
           \[
           \left\{
           \begin{aligned}
            \, \text{(i)}  & \, \text{all vertices in $\bm{\frF}(x,L-1)$ are $\phi$-open},\\
            \text{(ii)} & \, \text{there exists some vertex $y$ in $\bm{\frF}(x,L-1)$ such that $y\in \cC_\infty^\bfF$},\\
            \text{(iii)} &\, \text{it satisfies $\sup \{S_y: y\in \frP_k^{x,L} \} \leq k M(L)$} 
           \end{aligned}\right.\] 
           (a plate that is not $\omega$-good is called \textit{$\omega$-bad}).
        \end{minipage}
    \end{equation}

    \begin{equation}\label{eq:u-vacant}
        \begin{minipage}{0.9\textwidth}
            We say a $\phi$-open vertex $x\in \bfZ^d$ is \emph{$u$-vacant} if $\cC^\vee(Q_k(x))\cap \cI^u_\omega=\varnothing$ and \emph{$u$-occupied} otherwise.
        \end{minipage}
    \end{equation}
\end{definition}
 Note that in general, if a vertex $x\in \bfZ^d$ (not necessarily $\phi$-open) fails to be $u$-vacant, it is either both $\phi$-open and $u$-occupied, or it is $\phi$-closed. \medskip

We aim to show that $\bbP^I_\omega$-a.s., the set of $u$-vacant vertices contains an infinite connected component in $\cC_\infty^\bfF$ for small $u>0$. To this end, a version of the renormalization scheme described in Section~\ref{sec:Quenched-Unif-Estimates} will be applied to the ``coarse-grained'' plane $\bfF$. Due to this preliminary coarse-graining, we identify the plane $\bfF$ with $\bbZ^2$ and denote the renormalized lattice at scale $n$ (recall $L_n=L_0\ell_0^n$) by
\begin{equation}
    \bfL_n= L_n \bfF ~(=kL_n \bbF) 
\end{equation}
in place of $\bbL_n$ in~\eqref{eq:L_n}.

\begin{lemma}\label{lem:verify-Z-sec5}
	Let $L_0 \geq 1$, $\ell_0 \geq 100$,  $k_{L_0} = \lceil (\log L_0)^2 \rceil \geq k_0(p,d)$ and for $x\in \bfL_0$ define 
    \begin{equation}
        \cZ_x =\IND \big\{\text{$\bm{\frP}(x,L_0)$ is $\omega$-bad} \big\}.
    \end{equation}

    Then $(\cZ_x)_{x\in \bfL_0}$ is $(k_{L_0} L_0,\ell_0)$-approximable with decay exponent $\eta_{\mathrm{hk}}$. Moreover, the approximating fields $(\mathcal{Z}_x^{(n)})_{x \in \bfL_0, n \geq 0}$ can be chosen so that $\bbQ_p[\mathcal{Z}_0^{(0)}=1] \leq \widehat{\vartheta}(L_0)$ for a sequence of positive real numbers $( \widehat{\vartheta}(L))_{L \geq 1}$ with $\lim_{L \to \infty} \widehat{\vartheta}(L) = 0$.
\end{lemma}

\begin{proof}
    We write $k=k_{L_0}$ for ease of notation. Recalling~\eqref{eq:SxL}, for $n\geq 0$, $x\in \bfL_0$, we define  
     \begin{equation}\label{eq:cZ_x^n}
        \left\{
        \begin{aligned}
        \, &\overline{A}_x = \{\text{there exists a $\phi$-closed vertex in $\bm{\frF}(x,L_0-1)$}\}, \\
        &\overline{B}_{n,x} = \Big\{\bm{\frF}(x,L_0-1) \overset{\phi}{\centernot\longleftrightarrow} \bm{\frF}(x,2L_n) \text{ in } \bfF \Big\}, \\
        &\overline{C}_{n,x} = \Big\{\sup_{y\in \frP_{k}^{x,L_0}} S_y^{(kL_n)}>k M(L_0)\Big\}, \\
        & \cZ_x^{(n)} = \IND_{\overline{A}_x\cup \overline{B}_{n,x}\cup \overline{C}_{n,x}}.
      \end{aligned}   \right. 
    \end{equation}
    We begin by verifying that $(\cZ_x^{(n)})_{x\in \bfL_0,n\geq 0}$ satisfies Assumption~\ref{ass:local-rv}. Due to the translation invariance of $\bbQ_p$, the random variables $(\mathcal{Z}_x^{(n)})_{x \in \bfL_0}$, $n \geq 0$, are identically distributed.
    Note that 
    \begin{equation}
       \{\bm{\frF}(x,L_0-1) \overset{\phi}{\longleftrightarrow} \bm{\frF}(x,2L_n) \text{ in } \bfF\} = \{\bm{\frF}(x,L_0-1) \overset{\phi}{\longleftrightarrow} \bm{\frF}(x,2L_n) \text{ in } \bm{\frP}(x,2L_n)\} ,
        \end{equation}
    so that $\overline{B}_{n,x}$ is $\cF_{E(B(x,3kL_n))}$-measurable. The $\cF_{E(B(x,3kL_n))}$-measurability of $\overline{C}_{n,x}$ follows similarly as in the proof of Lemma~\ref{lem:verify-Z-x} by \eqref{eq:LQ-measurable}. It remains to verify~\eqref{eq:stable}. For $x\in \bfL_0$, let 
\begin{equation}
    \left\{
    \begin{aligned}
        \, & E_{n,x}^{(1)} = \bigcap_{y\in \frP_k^{x,L_0}} \{S_y \leq kL_n\},\\
        & E_{n,x}^{(2)} =\left\{
        \begin{aligned}
           &  \text{if a $\phi$-open cluster in $\bfF$ intersects
         both $\bm{\frF}(x,L_0-1)$}\\
         &\text{and $\bm{\frF}(x,2L_n)$, then this cluster must be infinite}
        \end{aligned}\right\},
        \\
        &E_{n,x}^{\cZ}=E_{n,x}^{(1)} \cap E_{n,x}^{(2)}.
    \end{aligned}\right.
\end{equation}
    Note that $\overline{A}_x$ does not depend on $n$. It is also clear that $\mathbbm{1}_{\overline{B}_{n,x}}=\mathbbm{1}_{\overline{B}_{n+1,x}}$ under $E_{n,x}^{(2)}$. In addition, similar to the reasoning for~\eqref{eq:S_x-stable}, 
    \begin{equation}\label{eq:S_x-stable-2}
       S_y^{(kL_n)}\IND_{E_{n,x}^{(1)}} = S_y^{(kL_{n+1})}\IND_{E_{n,x}^{(1)}} = S_y \IND_{E_{n,x}^{(1)}},
    \end{equation}
    for all $y\in \frP_k^{x,L_0}$ and hence $\mathbbm{1}_{\overline{C}_{n,x}}= \mathbbm{1}_{\overline{C}_{n+1,x}}$ under $E_{n,x}^{(1)}$.  Therefore, 
    \begin{equation}
        \cZ_x^{(n)} \IND_{E_{n,x}^{\cZ}} = \cZ_x^{(n+1)} \IND_{E_{n,x}^{\cZ}}. 
    \end{equation}

    It remains to control the decay rate of $\bbQ_p[(E_{n,x}^\cZ)^c]$. By~\eqref{eq:S_x-decay}, 
    \begin{equation}\label{eq:E^1-decay}
        \bbQ_p[(E_{n,x}^{(1)})^c] \leq c k^{d}L_0^2 \exp(-c'(kL_n)^{\eta_{\mathrm{hk}}}) \leq c \exp(-c'(kL_n)^{\eta_{\mathrm{hk}}}). 
    \end{equation}
    We turn to $\bbQ_p[(E_{n,x}^{(2)})^c]$. If $(E_{n,x}^{(2)})^c$ occurs, then there exists a finite $\phi$-open cluster with diameter larger than $L_n$. By planar duality, this shows the existence of a $*$-connected $\phi$-closed circuit with diameter larger than $L_n$. We can then pick $cL_n$ vertices in the circuit with mutual $|\,\cdot\,|_{1,\bfZ^d}$-distance greater than $2$ (so that they are independent). Note that by \eqref{eq:RQ-decay}, $\bbQ_p[\phi_z=0]\leq c_{17}\exp({-c_{18}k})$. Therefore,  
    \begin{equation}\label{eq:E^2-decay}
        \bbQ_p[(E_{n,x}^{(2)})^c] \leq 8^{L_n} (c_{17}\exp({-c_{18}k}))^{cL_n} \leq c\exp(-c'kL_n).
    \end{equation}

    Combining~\eqref{eq:E^1-decay} and \eqref{eq:E^2-decay}, we obtain
    \begin{equation}\label{eq:E^cZ-dacay}
        \bbQ_p[(E^\cZ_{n,x})^c] \leq c\exp(-c'(kL_n)^{\eta_{\mathrm{hk}}}).
    \end{equation} 
    which completes the proof that $(\cZ_{x}^{(n)})$ satisfies Assumption \ref{ass:local-rv}. 

    We are left to compare $\cZ_x^{(n)}$ with $\cZ_x$ and show~\eqref{eq:def-appro-2} and~\eqref{eq:def-appro-3}. Note that under $E_{n,x}^\cZ$, by \eqref{eq:S_x-stable-2}, we have $S_y^{(kL_n)}=S_y$ and by the uniqueness of $\cC_\infty^\bfF$ (see Lemma~\ref{lem:percolation-plane}), 
    \begin{equation}
        \bm{\frF}(x,L_0-1) \overset{\phi}{\longleftrightarrow} \bm{\frF}(x,2L_n) \iff \bm{\frF}(x,L_0-1) \cap \cC_\infty^\bfF \neq \varnothing.
    \end{equation}
    Hence 
    \begin{equation}
        \cZ_x \IND_{E_{n,x}^\cZ} = \cZ_x^{(n)} \IND_{E_{n,x}^\cZ} . 
    \end{equation}
    Together with \eqref{eq:E^cZ-dacay}, we have verified that $(\cZ_x)_{x\in \bbL_0}$ satisfies Definition \ref{def:approximable} (taking $F_{n,x}^{\cZ}=E_{n,x}^{\cZ}$ in the definition). 

    It remains to prove that the probability of $\{\mathcal{Z}_0^{(0)}=1\}$ can be bounded from above by a quantity that goes to zero as $L_0\to\infty$. We have that (note that $k=k_{L_0} \geq (\log L_0)^2$) 
    \begin{equation}
        \bbQ_p[\overline{A}_0] \leq cL_0 \exp(-ck) \to 0, \quad \text{as $L_0\to\infty$}. 
    \end{equation}
    Next, we note that $\overline{B}_{0,x}$ implies that $\bm{\frF}(x,L_0-1)$ and $\bm{\frF}(x,2L_0)$ are not connected by an open path in the dominated Bernoulli site percolation with parameter $q^*(k)$. As we are in the supercritical regime of the latter when $L_0$ is large enough (since $q^*(k)\to 1$ as $L_0\to\infty$), we find
    \begin{equation}
        \bbQ_p[\overline{B}_{0,0}] \to 0, \quad \text{as $L_0\to\infty$}. 
    \end{equation}
    Finally by \eqref{eq:S_x-decay}
    \begin{equation}
        \bbQ_p[\overline{C}_{0,0}] \leq c k^{d-2} L_0^2 \exp(-c(kM(L_0))^{\eta_{\mathrm{hk}}}) \to 0, \quad \text{as $L_0\to\infty$}. 
    \end{equation}
    It follows by the definition in \eqref{eq:cZ_x^n} that $\bbQ_p[\mathcal{Z}_0^{(0)}=1] \leq \widehat{\vartheta}(L_0)$ with $\widehat{\vartheta}(L) \to 0$ as $L\to \infty$. 
\end{proof}

Recall that we identify $\bfF$ with $\bbZ^2$ and write $\bfL_n$ in place of $\bbL_n$. We also denote by $\Lambda_{n,x}^\bfF$ (instead of $\Lambda_{n,x}$) the set of all proper tree embeddings of $T_n$ rooted at $x\in \bfL_n$ into $\bfF$ (more precisely, $\Lambda_{n,x}^\bfF$ contains the proper tree embeddings defined as in Definition~\ref{def:Tree-embeddings}, but with $\bbL_n$ replaced by $\bfL_n$, and $|\, \cdot \,|_\infty$ in~\eqref{eq:T-m-1-2} replaced by $|\, \cdot \,|_{\infty,\bfF}$).  
We can apply Theorem~\ref{thm:general-proportion} with $(\zeta_x)_{x\in \bbL_0}=(\cZ_x)_{x\in \bfL_0}$ (as $\bbL_0$ corresponds to $\bfL_0$), $\rho =\frac{1}{2}$ and $D_n= \{x\in \bfL_n:|x|_{\infty,\bfZ^d}\leq 2L_{n+1}\}$ (note that $|D_n|\leq (4\ell_0+1)^2$ satisfies~\eqref{eq:window-growth}), and obtain the following statement, in the same way as in the proof of Corollary~\ref{coro:subcritical}. 

\begin{corollary}\label{coro:proportion-F}
    Let $\ell_0 \geq 3^{1/\eta_{\mathrm{hk}}}\vee 100$ be a fixed integer. There exists $L_*(\ell_0)<\infty$ such that for $L_0\geq L_*(\ell_0)$ and a set $\widetilde{\Omega}_2\subseteq \Omega_2$ with $\bbQ_p[\widetilde{\Omega}_2]=1$, for $\omega\in \widetilde{\Omega}_2$ and $x\in \bfL_n$ with $|x|_{\infty, \bfZ^d} \leq 2L_{n+1}$, for $n\geq n_0(\omega,\ell_0,L_0)$ large enough and any choice of $\cT\in \Lambda_{n,x}^\bfF$, 
    \begin{equation}
        \sum_{m\in T_{(n)}} \cZ_{\cT(m)} \leq 2^{n-1}. 
    \end{equation}
\end{corollary}

We take the following bound from \cite[Remark~2.2]{prevost2025}, which will be used in the proof: 
    \begin{equation}\label{eq:capacity-thickened-frame}
        \capa\big(\frF_k^{0,L_0}\big)
        \leq
        \left\{
        \begin{aligned}
             &\frac{c_{19}kL_0}{\log L_0}, & d=3,\\
            &c_{19}k^{d-2}L_0, & d\geq 4.
        \end{aligned}\right.
    \end{equation}

\begin{proposition}\label{prop:infinite-in-renormalized}
    For $\omega \in \widetilde{\Omega}_2$, there exists $u_0(p, d)>0$ such that for $0<u<u_0(p, d)$, 
    \begin{equation}\label{eq:infinite-cluster-1}
        \bbP^I_\omega \Big[\text{$u$-vacant vertices form an infinite connected component in $\cC_\infty^\bfF$}\Big] =1 . 
    \end{equation}
\end{proposition}

\begin{proof}
    Throughout the proof, we fix 
    \begin{equation}\label{eq:l0L0k}
        \begin{split}
            &\ell_0 \geq 3^{1/\eta_{\mathrm{hk}}} \vee 100, \quad \\
            &L_0 \geq L_*(\ell_0) \vee \min \{L \geq 1: c_8 M(\tilde{L}) \leq \tilde{L},\text{ for all } \tilde{L}\geq L\} \\
            & \qquad \vee \exp((2c_{19}\cdot c_{10}) \vee \sqrt{k_0(p,d)}), \text{(with $k_0(p,d)$ in~\eqref{eq:k0pd}, $L_*(\ell_0)$ in Corollary~\ref{coro:proportion-F})} \\
            & k=k_{L_0} = \lceil (\log L_0)^2 \rceil  \ (\geq k_0(p,d)). 
        \end{split}
    \end{equation}

    Suppose that the event under the probability in \eqref{eq:infinite-cluster-1} occurs, then the $u$-vacant vertices also form an infinite connected component in $\bfF$. We define for $j \geq 1$ 
    \begin{equation}
        A_j^u  = \Bigg\{ \begin{minipage}{0.7\textwidth} there exists an infinite nearest-neighbor path in $\bfF$ of $u$-vacant vertices starting from a vertex in $\bm{\frF}(0,L_j)$\end{minipage} \Bigg\}. 
    \end{equation}
    and note that it suffices to show that $\lim_{j\to\infty} \bbP^I_\omega[(A_j^u)^c]= 0$. If $(A_j^u)^c$ occurs, then by planar duality of $\bfF$, there exists a $*$-connected circuit surrounding $\bm{\frF}(0,L_j)$ consisting of vertices that are not $u$-vacant (recall~\eqref{eq:u-vacant} and the line below it). 
   We proceed as in the proof of~\cite[Corollary 5.2]{rath2015transition}, and define for $n\geq 1$ and $x\in \bfL_n$ the events
    \begin{equation}
        B_{n,x}^u=
  \Bigg\{
  \begin{minipage}{0.7\textwidth}
    there exists a $*$-connected path $\bm{\gamma}$ of vertices in $\bfF$, with each vertex in $\bm{\gamma}$ not $u$-vacant, that connects $\bm{\frF}(x, L_n-1)$ and $\bm{\frF}(x, 2L_n)$
  \end{minipage}
\Bigg\}.
    \end{equation}
    With this, we obtain
    \begin{equation}\label{eq:A_n-union-bound}
        \bbP^I_\omega [(A_j^u)^c] \leq \sum_{n=j}^\infty \sum_{x\in \bfL_n, |x|_{\infty,\bfZ^d} \leq 2L_{n+1}} \bbP^I_\omega[B_{n,x}^u]. 
    \end{equation}

    Our goal is to find a uniform upper bound for $\bbP^I_\omega[B_{n,x}^u]$ over $x \in \bfL_n$ with $|x|_{\infty, \bfZ^d}\leq 2L_{n+1}$. By Lemma~\ref{lem:renormalization}, under $B^u_{n,x}$, there exists $\cT\in {\Lambda}^{\bfF}_{n,x}$ such that for all $m \in T_{(n)}$, 
    \begin{equation}
         \bm{\gamma}\cap \bm{\frF}(\cT(m), L_0-1) \neq \varnothing. 
    \end{equation}
    By Corollary \ref{coro:proportion-F}, we know that for all $\cT\in {\Lambda}^{\bfF}_{n,x}$, at least $2^{n-1}$ of the plates $(\bm{\frP}(\mathcal{T}(m),L_0))_{m \in T_{(n)}}$, are $\omega$-good. We denote the index set of those plates by $\widehat{T}_{(n),\cT}$, which is a subset of $T_{(n)}$ (with $|\widehat{T}_{(n),\cT}| \geq 2^{n-1}$). For $m\in \widehat{T}_{(n),\cT}$, all vertices in $\bm{\frF}(\cT(m), L_0-1)$ are $\phi$-open (see~\eqref{eq:plate-good}~(i)) so that $\bm \gamma \cap \bm{\frF}(\cT(m), L_0-1)$ is a non-empty set of $u$-occupied vertices (recall the definition~\eqref{eq:u-vacant}). Using again~\eqref{eq:u-vacant}, we have 
    \begin{equation}
        \cI^u_\omega \cap \bigcup_{y\in \bm{\frF}(\cT(m), L_0-1)} \cC^\vee(Q_k(y)) \neq \varnothing, \quad \text{for all } m\in \widehat{T}_{(n),\cT}. 
    \end{equation}
    The above gives
    \begin{equation}
        \bbP^I_\omega[B_{n,x}^u] \leq \bbP^I_\omega\Bigg[\bigcup_{\cT\in \Lambda_{n,x}^{\bfF}} \bigcap_{m\in \widehat{T}_{(n),\cT}} \Big\{\cI^u_\omega \cap \bigcup_{y\in \bm{\frF}(\cT(m), L_0-1)} \cC^\vee(Q_k(y)) \neq \varnothing\Big\} \Bigg]. 
    \end{equation}
    
    For $\cT\in \Lambda_{n,x}^{\bfF}$, $m\in \widehat{T}_{(n),\cT}$, we define by 
    \begin{equation}\label{eq:porous-frame}
        \begin{split}
            &\widehat{\frF}_k(\cT(m)) = \bigcup_{y\in \bm{\frF}(\cT(m), L_0-1)} \cC^\vee(Q_k(y)) ~(\subseteq \frF_k^{\cT(m),L_0}), \quad \cX_\cT = \bigcup_{m\in\widehat{T}_{(n),\cT}} \widehat{\frF}_k(\cT(m)), \\
            & (\text{$\cX_\cT \subseteq \cC_\infty$ by \eqref{eq:plate-good}~(ii) and \eqref{eq:subset-C-infty}})
        \end{split}
    \end{equation} 
    the \emph{thickened porous frame} at $\mathcal{T}(m)$, and the union over all these frames, respectively.

    Recall the notation $N_{K}$ for $K\subseteq \cC_\infty$, $(X^j)_{j \geq 1}$, and $\mathsf{P}^\omega$ from the constructive definition of random interlacements in Lemma~\ref{lem:constructive-RI}. For a simple random walk $X$, we denote by 
    \begin{equation}
        \cN(X) = \sum_{m\in \widehat{T}_{(n),\cT}} \IND \big\{ \mathrm{range}( X) \cap \widehat{\cF}_k(\cT(m)) \neq \varnothing \big\} 
    \end{equation}
    the number of thickened porous frames visited by $X$. 
    Then 
    \begin{equation}
        \bbP^I_\omega \bigg[\bigcap_{m\in \widehat{T}_{(n),\cT}} \Big\{\cI^u_\omega \cap \widehat{\frF}_k(\cT(m))\neq \varnothing\Big\} \bigg] \leq \mathsf{P}^\omega \bigg[\sum_{j=1}^{N_{\cX_\cT}} \cN(X^j) \geq  2^{n-1} \bigg]. 
    \end{equation}
    Our goal is then to stochastically dominate $\cN(X)$ by a geometrically distributed random variable. Note that $\cZ_x=0$ implies that $S_z \leq kM(L_0)$ for all $z\in \frP_k^{x,L_0}$ by~\eqref{eq:plate-good}~(iii).
    For $m\in \widehat{T}_{(n),\cT}$ and $m' \in T_{(n)}^{m,j}$ for $1\leq j \leq n$, Lemma~\ref{lem:tree-geometry} yields that $|\cT(m)-\cT(m')|_\infty \geq 90kL_j$. Therefore by~\eqref{eq:off-diag-Green} and using~\eqref{eq:l0L0k}, one has
    \begin{equation}\label{eq:sec-5-Green}
    g^\omega(y,z) \leq \frac{c_{10}}{(kL_{j})^{d-2}}, \quad \text{for $y\in \frP_k^{\cT(m),L_0}$ and $z\in \frP_k^{\cT(m'),L_0}$}.  
    \end{equation}

    We then adapt the proof of Proposition~5.1 in \cite{rath2015transition} to our setup.
    For $y\in \widehat{\frF}_k(\cT(m))$, 
    \begin{equation}
         \begin{split}
        P_y^\omega [\mathrm{range}(X) &\cap \cX_\cT\setminus \widehat{\frF}_k(\cT(m)) \neq \varnothing] 
        \leq \sum_{j=1}^n \sum_{m' \in T_{(n)}^{m,j} \cap  \widehat{T}_{(n),\cT}} P_y^\omega [\mathrm{range}(X) \cap \widehat{\frF}_k(\cT(m')) \neq \varnothing]\\
        &\underset{\eqref{eq:sec-5-Green}}{\overset{\eqref{eq:FED}}{\leq} }\sum_{j=1}^n \sum_{m' \in T_{(n)}^{m,j} \cap  \widehat{T}_{(n),\cT}} c_{10} (kL_{j})^{2-d} \capa^\omega(\widehat{\frF}_k(\cT(m')) \\
        &\underset{\eqref{eq:capacity-monotone}}{\overset{\eqref{eq:porous-frame}}{\leq}} \sum_{j=1}^n \sum_{m' \in T_{(n)}^{m,j} \cap  \widehat{T}_{(n),\cT}} c_{10} (kL_j)^{2-d} \capa(\frF_k^{0,L_0}) \\
        &\underset{\eqref{eq:sub-division-tree}}{\overset{\ell_0\geq 100}{\leq}} c_{10} (kL_0)^{2-d}  \capa(\frF_k^{0,L_0})  \overset{\eqref{eq:capacity-thickened-frame}}{\leq }
        \left\{
        \begin{aligned}
            \dfrac{c_{19}\cdot c_{10}}{\log L_0}, \quad &d=3,\\
            \dfrac{c_{19}\cdot c_{10}}{L_0^{d-3}},\quad & d\geq 4.
        \end{aligned}
        \right. 
    \end{split}
    \end{equation}
We define
    \begin{equation}
   \theta_d(L_0) \overset{\mathrm{def}.}{=} \left\{
        \begin{aligned}
            \dfrac{c_{19}\cdot c_{10}}{\log L_0}, \quad &d=3,\\
            \dfrac{c_{19}\cdot c_{10}}{L_0^{d-3}},\quad & d\geq 4.
        \end{aligned}
        \right.
    \end{equation}
    and note that $\theta_d(L_0) < \frac{1}{2} \text{ by~\eqref{eq:l0L0k}}$.
    Abbreviating $\theta_d(L_0)$ by $\theta$, we then use the strong Markov property of the random walk to conclude that $\cN(X)$ is dominated by a geometric random variable with parameter $\theta$, so $E^\omega_{\widetilde{e}_{\cX_\cT}}[z^{\cN(X)}]\leq \frac{(1-\theta)z}{1-\theta z}$ for any $1\leq z<\frac{1}{\theta}$. Since $N_{\cX_\cT}$ follows a Poisson distribution with parameter $u\capa^\omega(\cX_\cT)$, we see that 
    \begin{equation}\label{eq:exponential-moment}
        \begin{split}
            \mathsf{E}^\omega &\big[ z^{\sum_{j=1}^{N_{\cX_\cT}} \cN(X^j)} \big] \\
            &= \exp\Big(u\capa^\omega(\cX_\cT)\cdot \big(E^\omega_{\widetilde{e}_K}[z^{\cN(X)}]-1 \big)  \Big) 
        \leq \exp\Big(u\capa^\omega(\cX_\cT)\cdot \frac{z-1}{1-\theta z}  \Big).
        \end{split}
    \end{equation}
    We then apply the exponential Chebyshev inequality with $z=\frac{1}{2\theta}$ ($>1$): 
    \begin{equation}
    \begin{split}
        \bbP^I_\omega &[B_{n,x}^u] \leq (\cC_{2,\ell_0})^{2^n} \mathsf{E}^\omega \bigg[\Big(\frac{1}{2\theta}\Big)^{
        \sum_{j=1}^{N_{\cX_\cT}} \cN(X^j)}\bigg] (2\theta)^{2^n}  \\
        & 
        \overset{\eqref{eq:exponential-moment}}{\leq}
         \exp\Big(u\capa^\omega(\cX_\cT)\cdot\frac{1}{\theta} \Big) \cdot (2\theta \cC_{2,\ell_0})^{2^n}
        \underset{\eqref{eq:capacity-monotone}}{\overset{\eqref{eq:porous-frame}}{\leq}} \exp(u\cdot \capa(\frF_k^{0,L_0}) \cdot \cC_{2,\ell_0})^{2^n} \\
       & \leq \exp(u\cdot c_{19} k^{d-2} L_0 \cdot \cC_{2,\ell_0})^{2^n}. 
           \end{split} 
    \end{equation}
    Recall that $k=k_{L_0}$. If $u_0=(2c_{19} k^{d-2} L_0 \cdot \cC_{2,\ell_0})^{-1}$ and $0<u<u_0$, then by~\eqref{eq:A_n-union-bound} $\lim_{j\to\infty}\bbP_\omega^I[(A_j^u)^c]=0$, which completes the proof. 
\end{proof}

With Proposition~\ref{prop:infinite-in-renormalized} in place, we now show that the vacant set percolates on $\mathcal{C}_\infty$ due to the good connectivity property of $u$-vacant boxes. 

\begin{proof}[\textbf{Proof of Theorem~\ref{thm:PH-RI}, Part II}]
    We have shown in the first part of the proof that $u_*^\omega$ equals a deterministic constant $u_*$ with $u_*<\infty$, see \eqref{eq:u-star-finite}. \medskip
    
    We now show that $u_*>0$. 
    Take $u_0(p,d)$ to be the quantity in Proposition~\ref{prop:infinite-in-renormalized} and $\widetilde{\Omega}_2 \subseteq \Omega_0$ such that Proposition~\ref{prop:infinite-in-renormalized} holds. Let $0<u<u_0(p,d)$. For $\omega\in \widetilde{\Omega}_2$ and $\eta$ an element belonging to the event under the probability in \eqref{eq:infinite-cluster-1}, let $\bm{\cD}^u(\omega, \eta) \subseteq \bfF$ be the infinite connected component in $\cC_\infty^\bfF$ of $u$-vacant vertices. We claim that 
    \begin{equation}\label{eq:claim-infinite}
        \text{$\bigcup_{x\in \bm{\cD}^u} \cC^\vee(Q_k(x))$ is an infinite connected subset of $\cV^u_\omega$}.
    \end{equation}
    We first see from~\eqref{eq:subset-C-infty} that $\bigcup_{x\in \bm{\cD}^u} \cC^\vee(Q_k(x)) \subseteq \cC_\infty$, and by~\eqref{eq:C^V-connected} we obtain the connectedness by open paths. The latter set contains an unbounded path since $\bm{\cD}^u$ is infinite. Finally, by the definition of $u$-vacancy in~\eqref{eq:u-vacant}, $\cC^\vee(Q_k(x)) \subseteq \cV^u_\omega$ for $x\in \bm{\cD}^u$. The claim~\eqref{eq:claim-infinite} therefore holds. This shows that for $\bbP^I_\omega$-a.e.~$\eta$, $\cV^u_\omega$ percolates in a slab and thus $u_*>0$.  
\end{proof}

We now briefly comment on further applications and potential extensions of our main results.

\begin{remark}
\label{rem:Final}
    \begin{enumerate}
        \item The method introduced in Theorem~\ref{thm:general-proportion} should be pertinent for establishing several variants of the main results, Theorem~\ref{thm:PH-GFF} and Theorem~\ref{thm:PH-RI}, as we now explain. For instance, one might replace $\Omega_{\mathrm{bond}} = \{0,1\}^{\bbE_d}$ by a set $\Omega_{\mathrm{bond}}^{\lambda_{\mathrm{min}}} = (\{0\} \cup [\lambda_{\mathrm{min}},1])^{\bbE_d}$ for some $\lambda_{\mathrm{min}} \in (0,1)$ and consider a probability measure $\widetilde{\mathbb{Q}}$ on $\Omega_{\mathrm{bond}}^{\lambda_{\mathrm{min}}}$  under which the canonical coordinates $(\widetilde{\omega}_e)_{e \in \bbE_d}$ are i.i.d.~and fulfill
        \begin{equation}
        \widetilde{\mathbb{Q}}[\widetilde{\omega}_e > 0] \in (p_c(d),1)
        \end{equation}
(this corresponds for instance to the assumptions in~\cite[Theorem 1.9]{schweiger2024maximum}, but for $d \geq 3$ in our case). The corresponding weighted graph $(\mathcal{C}_\infty, \bbE_{\mathcal{C}_\infty},\omega_e)$ (with $\mathcal{C}_\infty$ the infinite connected component of edges with positive conductance) can then similarly be equipped with measures $\mathbb{P}^G_{\widetilde{\omega}}$ or $\mathbb{P}^I_{\widetilde{\omega}}$. Since the stochastic integrability~\eqref{eq:S_x-decay} and the heat kernel bounds~\eqref{eq:Barlow-HK} remain valid in this set-up (see also Remark 4.14 (2) in~\cite{chiarini2026solidification} for a similar discussion), our main results Theorem~\ref{thm:PH-GFF} and Theorem~\ref{thm:PH-RI} provide non-trivial percolation phase transitions for $E^{\geq \alpha}_{\widetilde{\omega}}$ and $\mathcal{V}^u_{\widetilde{\omega}}$, for $\widetilde{\mathbb{Q}}$-a.e.~$\widetilde{\omega} \in \Omega_{\mathrm{bond}}^{\lambda_{\mathrm{min}}}$. In a different direction, it further seems plausible that Theorem~\ref{thm:general-proportion} should be helpful in proving phase transitions for other percolation models with long-range correlations on $\mathcal{C}_\infty$, upon combination with an appropriately chosen static renormalization scheme. 
\item For the level sets of the GFF and the vacant set of random interlacements on $\mathbb{Z}^d$, $d \geq 3$, a very precise understanding has emerged for the nature of the phase transition (see~\cite{drewitz2018sign,hugo2023equality,
duminilcopin2023phasetransitionvacantset,
goswami2025stronglocaluniquenessvacant}) and the off-critical behavior of both models (see~\cite{chiarini2020GFF,
chiarini2020entropic,goswami2022radius,
li2014,
mcauley2025limit,
nitzschner2017solidification,
sznitman2015disconnection,
sznitman2017, sznitman2019macroscopic}). In particular, we highlight the \textit{sharpness} of the phase transition for both models (see~\cite{hugo2023equality} for the level-sets of the GFF and~\cite{duminilcopin2023phasetransitionvacantset,
goswami2025stronglocaluniquenessvacant} for the vacant set of random interlacements) and the fact that $\alpha_\ast^{\mathbb{Z}^d} > 0$ for the level sets of the GFF in~\cite{drewitz2018sign} (with $\alpha_\ast^{\mathbb{Z}^d}$ the percolation threshold for the GFF on $\mathbb{Z}^d$, $d \geq 3$). In fact, for each model, additional critical parameters $\overline{\alpha}^{\mathbb{Z}^d}$ and $\alpha_{\ast\ast}^{\mathbb{Z}^d}$ with the (a priori) property $\overline{\alpha}^{\mathbb{Z}^d} \leq {\alpha}_\ast^{\mathbb{Z}^d} \leq {\alpha}_{\ast\ast}^{\mathbb{Z}^d} $, or $\overline{u}^{\mathbb{Z}^d}$ and $u_{\ast\ast}^{\mathbb{Z}^d}$ with $ \overline{u}^{\mathbb{Z}^d} \leq {u}_\ast^{\mathbb{Z}^d} \leq {u}_{\ast\ast}^{\mathbb{Z}^d}$ have in the past been used to study various phenomena in the off-critical regimes of these models (with the sharpness of the phase transition corresponding to the series of equalities $ \overline{\alpha}^{\mathbb{Z}^d} = {\alpha}_\ast^{\mathbb{Z}^d} = {\alpha}_{\ast\ast}^{\mathbb{Z}^d}$ and $\overline{u}^{\mathbb{Z}^d} = {u}_\ast^{\mathbb{Z}^d} = {u}_{\ast\ast}^{\mathbb{Z}^d}$). One can introduce equivalent parameters $\overline{\alpha}^G, \alpha^G_{\ast\ast}$, $\overline{u}^G$, and $u^G_{\ast\ast}$ in a very general context of locally finite, connected, transient weighted graphs $G$ with some \textit{global} regularity conditions $(p_0)$, $(V_\alpha)$, $(G_\beta)$, and $(\mathrm{WSI})$ as in~\cite{drewitz2025} (see also~\cite{sznitman2012decoupling} for a similar set-up for graphs of the form $\widetilde{G} \times \mathbb{Z}$). One central result of~\cite{drewitz2025} is that under these general conditions, one has $\overline{\alpha}^G > 0$ (and, a fortiori, $\alpha^G_\ast > 0$, with hopefully obvious notation). \medskip

\noindent
We highlight that a typical realization of the percolation cluster $\mathcal{C}_\infty$ does \textit{not} fulfill the criteria $(V_\alpha)$, $(G_\beta)$, which essentially correspond to a global $\alpha$-Ahlfors regularity for the measure of balls and global controls on the Green function for the simple random walk on $G$ (see~\cite[Section 1]{drewitz2025} for the precise definitions). In fact, due to the global nature of the definitions of quantities such as $\overline{\alpha}^G, \alpha^G_{\ast\ast}$, $\overline{u}^G$, and $u^G_{\ast\ast}$ in~\cite{drewitz2025}, an adaptation to our set-up (and therefore the definition of corresponding strongly supercriticial or strongly subcritical regimes for $E^{\geq \alpha}_\omega$ and $\mathcal{V}^u_\omega$) is not entirely obvious, and would have to take into account local degeneracies of $\mathcal{C}_\infty$. We refer to Remark 3.7 (2) in~\cite{CN2021disconnection} for a similar discussion pertaining to the level-sets of the GFF on $\mathbb{Z}^d, d \geq 3$, with uniformly elliptic bond disorder.  Quite plausibly, the extraction methods in Section~\ref{sec:Quenched-Unif-Estimates} might be useful to obtain analogues of the parameters  $\overline{\alpha}^G, \alpha^G_{\ast\ast}$, $\overline{u}^G$, and $u^G_{\ast\ast}$ for a typical realization of $\mathcal{C}_\infty$. The methods of the present article also serve as a starting point to study the off-critical regimes of $E^{\geq \alpha}_\omega$ and $\mathcal{V}^\alpha_\omega$, in particular in combination with recently obtained capacity controls on percolation clusters in~\cite{chiarini2026solidification}, which we plan to address elsewhere.
    \end{enumerate}
\end{remark}

\noindent
\textbf{Acknowledgements.} MN was partially supported by Hong Kong RGC grants ECS 26301824 and GRF 16303825. While this work was written, AC was associated with INdAM (Istituto Nazionale di Alta Matematica “Francesco Severi”) and the group GNAMPA.

\bibliographystyle{plain}

\end{document}

%% file: commands.tex
\newcommand{\IND}{\mathbbm{1}}

\newcommand{\capa}{\mathrm{Cap}}

\newcommand{\var}{\mathrm{Var}}

\newcommand{\cC}{\ensuremath{\mathcal C}}
\newcommand{\cD}{\ensuremath{\mathcal D}}
\newcommand{\cE}{\ensuremath{\mathcal E}}
\newcommand{\cF}{\ensuremath{\mathcal F}}

\newcommand{\cI}{\ensuremath{\mathcal I}}

\newcommand{\cL}{\ensuremath{\mathcal L}}

\newcommand{\cN}{\ensuremath{\mathcal N}}
\newcommand{\cO}{\ensuremath{\mathcal O}}

\newcommand{\cR}{\ensuremath{\mathcal R}}

\newcommand{\cT}{\ensuremath{\mathcal T}}

\newcommand{\cV}{\ensuremath{\mathcal V}}

\newcommand{\cX}{\ensuremath{\mathcal X}}

\newcommand{\cZ}{\ensuremath{\mathcal Z}}
\newcommand{\bbE}{\ensuremath{\mathbb E}}
\newcommand{\bbF}{\ensuremath{\mathbb F}}

\newcommand{\bbL}{\ensuremath{\mathbb L}}

\newcommand{\bbN}{\ensuremath{\mathbb N}}

\newcommand{\bbP}{\ensuremath{\mathbb P}}
\newcommand{\bbQ}{\ensuremath{\mathbb Q}}
\newcommand{\bbR}{\ensuremath{\mathbb R}}

\newcommand{\bbZ}{\ensuremath{\mathbb Z}}
\newcommand{\frF}{\ensuremath{\mathfrak F}}

\newcommand{\frP}{\ensuremath{\mathfrak P}}

\newcommand{\bfA}{\ensuremath{\mathbf A}}
\newcommand{\bfB}{\ensuremath{\mathbf B}}

\newcommand{\bfF}{\ensuremath{\mathbf F}}

\newcommand{\bfH}{\ensuremath{\mathbf H}}

\newcommand{\bfK}{\ensuremath{\mathbf K}}
\newcommand{\bfL}{\ensuremath{\mathbf L}}

\newcommand{\bfP}{\ensuremath{\mathbf P}}

\newcommand{\bfZ}{\ensuremath{\mathbf Z}}